\documentclass[11pt]{amsart}
\usepackage{geometry}                
\usepackage{graphicx}
\usepackage{amssymb}
\usepackage{epstopdf}
\newtheorem{theorem}{Theorem}[section]
\newtheorem{lemma}{Lemma}[section]
\newtheorem{corollary}{Corollary}[section]

\newtheorem{conjecture}{Conjecture}[section]
\newtheorem{definition}{Definition}[section]

\newtheorem{remark}{Remark}[section]
\newtheorem{question}{Question}[section]
\numberwithin{equation}{section}

\def\Z{\Bbb Z}

\def\R{\Bbb R}

\def\C{\Bbb C}
\def\A{\Bbb A}
\def\N{\Bbb N}

\def\D{\Delta}
\def\Om{\Omega}

\def\d{\partial}

\def\a{\alpha}
\def\b{\beta}
\def\g{\gamma}

\def\s{\sigma}
\def\e{\epsilon}
\def\om{\omega}

\title{Causal Holography of Traversing Flows} 
\author{Gabriel Katz}
\address{Massachusetts Institute of Technology,
Department of Mathematics,
77 Massachusetts Avenue,
Cambridge, MA 02139, U.S.A.}
\email{gabkatz@gmail.com}

\begin{document}

\maketitle

\begin{abstract} We study smooth {\sf traversing} vector fields $v$ on compact manifolds $X$ with boundary. A traversing $v$ admits a Lyapunov function $f: X \to \R$ such that $df(v) > 0$. 

We show that the trajectory spaces $\mathcal T(v)$ of {\sf traversally generic} $v$-flows are {\sf Whitney stratified spaces}, and thus admit triangulations amenable to their natural stratifications. Despite being spaces with singularities, $\mathcal T(v)$ retain some residual smooth structure of $X$. 

Let  $\mathcal F(v)$ denote the oriented $1$-dimensional foliation on $X$, produced by a traversing $v$-flow. With the help of a {\sf boundary generic} $v$, we divide the boundary  $\d X$ of $X$ into two complementary compact manifolds, $\d^+X(v)$ and $\d^-X(v)$.
 
Then, for a traversing $v$, we introduce the {\sf causality map} $C_v: \d^+X(v) \to \d^-X(v)$. Our main result claims that, for boundary generic traversing vector fields $v$, the causality map $C_v$ is allows for a reconstruction of the pair $(X, \mathcal F(v))$, up to a homeomorphism  $\Phi: X \to X$ such that $\Phi|_{\d X} = id_{\d X}$. In other words, for a massive class of ODEs, we show that the topology of their solutions, satisfying a given  boundary value problem, is {\sf rigid}. We call these results ``{\sf holographic}" since the $(n+1)$-dimensional $X$ and the un-parameterized  dynamics of the $v$-flow are captured by a single map $C_v$ between two $n$-dimensional screens, $\d^+X(v)$ and $\d^-X(v)$. 

This holography of traversing flows has numerous applications to the dynamics of general flows. Some of them are described in the paper. Others, are just outlined. 
 \end{abstract}

\section{Introduction}

This paper is an extension of the sequence \cite{K1} - \cite{K4}, which studies non-vanishing  gradient-like flows on smooth compact manifolds with boundary.  Our approach emphasizes the interactions of the flow trajectories with the boundary. \smallskip

Let $X$ be a compact connected smooth $(n+1)$-dimensional manifold with boundary. A smooth vector field $v$ on $X$ is called {\sf traversing} if each $v$-trajectory is homeomorphic either to a closed interval, or to a singleton. An equivalent definition of a traversing $v$ is based on the existence of a Lyapunov function $f: X \to \R$ such that $df(v) > 0$ in $X$. In particular, the gradient flow of a Bott-Morse function $f$ is traversing in the compliment to any open neighborhood of its critical set. 
\smallskip

The paper consists of five sections, including the Introduction. \smallskip

{\sf In Section 2}, we introduce various classes of vector fields on manifolds with boundary and summarize their properties, needed for the rest of the paper. They include {\sf traversing, boundary generic}, and {\sf traversally generic} vector fields.\smallskip

{\sf In Section 3}, we employ the semi-local algebraic models for boundary generic and traversally generic vector fields $v$ on $X$ to get a better understanding of the {\sf trajectory space} $\mathcal T(v)$ of the $v$-flow and its intricate stratification by the {\sf combinatorial types} of $v$-trajectories. These types $\om$ belong to an universal poset $\mathbf \Omega^\bullet$,  introduced in \cite{K3}.  They describe the tangency patterns of trajectories to the boundary $\d X$ and resemble the real divisors of real polynomials.  

For traversing flows, $\mathcal T(v)$, despite being singular spaces, retain some surrogate smooth structure (see Definition \ref{def2.2}), which they inherit from $X$. In fact, $\mathcal T(v)$ also shares with $X$ all stable characteristic classes of its surrogate ``tangent bundle" $\tau(\mathcal T(v))$. 
\smallskip

Theorem \ref{th2.2} is the main result of this section. It claims that, for a traversally generic vector field $v$, the trajectory space $\mathcal T(v)$ can be given the structure of {\sf Whitney stratified space} (see Definition \ref{def2.3}).  As a result, for a traversally  generic $v$, the trajectory space $\mathcal T(v)$ admits a triangulation, amenable to its $v$-flow-induced $\mathbf\Omega^\bullet$-stratification (Corollary \ref{cor2.4}).
Therefore, for such a $v$, the trajectory space $\mathcal T(v)$ is a $n$-dimensional compact $\mathbf\Omega^\bullet$-stratified $CW$-complex, homotopy equivalent to $X$ (Corollary \ref{cor2.4}). \smallskip 
Unfortunately, the proof of Theorem \ref{th2.2} is lengthy. The reader, interested only in the main result of the paper, may choose to proceed directly to Section 4. \smallskip

{\sf In Section 4}, we are preoccupied with the following central to our program question:

\emph{``For a traversing vector field $v$ on a compact connected manifold $X$, what kind of residual structure on its boundary $\d X$ allows for a reconstruction of the pair $(X, v)$, say, up to a homeomorphism or a diffeomorphism?"} 
\smallskip

If such a structure on the boundary is available, it deserves to be called {\sf holographic}, since the information about the $(n+1)$-dimensional $v$-dynamics is recorded on a pair of $n$-dimensional records, residing in $\d X$ .

For a traversing field $v$, with the dream of holography in mind,  we introduce the {\sf causality map} $C_v: \d_1^+X(v) \to \d_1^-X(v)$ that takes any point $x \in \d X$, where the field is directed inward of $X$, to the ``next" along the trajectory $\g_x$ point $C_v(x) \in \d X$; at $C_v(x)$ the vector field $v$ is directed outwards.  

In general, the causality map $C_v$ is a \emph{discontinuous} map, with a very particular types of discontinuity. It is this discontinuity that captures the essential topology of $X$!

$C_v$ plays a role somewhat similar to the one played by the classical Poincar\'{e} return map: continuous flow dynamics is reduced to a single map of a lower-dimensional slice (\cite{Te}).
\smallskip 

Let $v_1$ be a traversing  and {\sf boundary generic} (see Definition \ref{boundary_generic}) field on a manifold $X_1$, and let $v_2$ be a traversing  and boundary generic field on a manifold $X_2$, where $\dim(X_1) = \dim(X_2)$. We denote by $\mathcal F(v_i)$ the oriented $1$-dimensional foliation on the manifold $X_i$, produced by the traversing vector field $v_i$ ($i = 1, 2$). 

Theorem \ref{th3.HOLO}---the main result of this paper---claims that any smooth diffeomorphism $\Phi^\d: \d_1X_1 \to \d_1X_2$ which commutes with the causality maps $C_{v_1}$ and $C_{v_2}$, extends to a  homeomorphism (often a smooth diffeomorphism) $\Phi: X_1 \to X_2$. Moreover, $\Phi$ takes each $v_1$-trajectory to a $v_2$-trajectory, thus mapping the $v_1$-oriented $1$-dimensional foliation $\mathcal F(v_1)$ to the $v_2$-oriented foliation $\mathcal F(v_2)$. 

In other words, for a traversing and boundary generic $v$, the causality map $C_v$ allows for a reconstruction of the pair $(X, \mathcal F(v))$, up to a homeomorphism (Corollary \ref{cor3.9}).  So the topology of $X$ and the unparametrized $v$-flow dynamics are \emph{topologically rigid} for the given ``boundary conditions" $C_v: \d_1^+X(v) \to \d_1^-X(v)$. In many cases (perhaps, allways), the reconstruction of $(X, \mathcal F(v))$ is possible up to a smooth diffeomorphism. 
\smallskip

Theorem \ref{th3.HOLO} leads to a novel representation, described in Theorem \ref{th3.4}, of smooth $(n+1)$-manifolds $X$ with spherical boundary. The representation is based on a map $C_v: D^n_+ \to D^n_-$ from one $n$-dimensional ball to another, $n \geq 2$, and captures the topological type of $X$.\smallskip

This topological rigidity has a number of implications for general dynamical systems (which are not necessarily of the gradient type). We summarize them in Theorem \ref{th3.6},  {\sf The Causal Holography Principle}. Vaguely, it states that  the {\sf causality relation} on a generic event horizon $H$ in the space-time space of a given dynamical system determines the compact portion $X$ of the event space, bounded by $H$, and the evolution of the system in $X$, up to a homeomorphism of $X$ which is the identity on $H$.
\smallskip
 
{\sf In  Section 5}, we sketch some applications of the Holographic Causality Theorem \ref{th3.HOLO} to geodesic flows on compact Riemannian manifolds with boundary (Theorem \ref{th4.1}). They revolve around some classical inverse scattering problems and geodesic billiards, as described in \cite{K6} and \cite{K9}. 
\smallskip

Let us conclude this Introduction with one remark which describes a paradoxical tension in our results. On the one hand, the causality maps are typically discontinuous, and that property is their nature. On the other hand, our techniques require a high degree of differentiability of the structures on the boundary, the structures that make the Holography Theorems valid and meaningful.  We would love to understand better the paradox.   

 
\section{Trivia: traversing, boundary generic, and traversally generic vector fields }

For the reader convenience, we start with a review of some properties of vector fields on manifolds with boundary that will be essential for the rest of the paper. The relevant definitions and facts are borrowed from \cite{K1}-\cite{K4} and \cite{K7}. See \cite{K5} for a more relaxed description of our approach to flows on surfaces.
\smallskip

Let $X$ be a compact connected smooth $(n+1)$-dimensional manifold with boundary. 

\begin{definition}\label{traversing}
A vector field $v$ on $X$ is called {\sf traversing} if each $v$-trajectory is ether a closed interval, or a singleton. \hfill $\diamondsuit$
\end{definition} 

In particular, a traversing vector field does not vanish and is of the {\sf gradient type}, i.e., there exists a smooth Lyapunov function $f:X \to \R$ such that $df(v) > 0$ in $X$. Moreover, the converse is true: any non-vanishing gradient-type vector field is traversing \cite{K1}. 

We denote by $\mathcal V_{\mathsf{trav}}(X)$ the space of all traversing fields on $X$. 
\smallskip

For a vector field $v \in \mathcal V_{\mathsf{trav}}(X)$, its trajectory space $\mathcal T(v)$ is homology equivalent  to $X$ (Theorem 5.1, \cite{K3}).  Moreover,  for a traversing field $v$, the trajectory space $\mathcal T(v)$  has an interesting feature:  it comes equipped with a vector $n$-bundle $\tau(\mathcal T(v))$ which plays the role of  ``surrogate tangent bundle". 
\smallskip

Any smooth vector field $v$ on $X$, which does not vanish along the boundary $\d X$, gives rise to a partition $\d_1^+X(v) \cup \d_1^-X(v)$ of the boundary $\d X$ into  two sets: the locus $\d_1^+X(v)$, where the field is directed inward of $X$ or is tangent to $\d X$, and  $\d_1^-X(v)$, where it is directed outwards or is tangent to $\d X$. 

We assume that $v|_{\d X}$, viewed as a section of the quotient  line bundle $T(X)/T(\d X)$ over $\d X$, is transversal to its zero section. This assumption implies that both sets $\d^+_1 X(v)$ and $\d^-_1 X(v)$ are compact manifolds which share a common boundary $\d_2X(v) =_{\mathsf{def}} \d(\d_1^+X(v)) = \d(\d_1^-X(v))$. Evidently, $\d_2X(v)$ is the locus where $v$ is \emph{tangent} to the boundary $\d X$.

Morse has noticed (\cite{Mo}) that, for a generic vector field $v$, the tangent locus $\d_2X(v)$ inherits a similar structure in connection to $\d_1^+X(v)$, as $\d X$ has in connection to $X$. That is, $v$ gives rise to a partition $\d_2^+X(v) \cup \d_2^-X(v)$ of  $\d_2X(v)$ into  two sets: the locus $\d_2^+X(v)$, where the field is directed inward of $\d_1^+X(v)$ or is tangent to $\d_2X(v)$, and  $\d_2^-X(v)$, where it is directed outward of $\d_1^+X(v)$ or is tangent to $\d_2X(v)$. Again, we assume that $v|_{\d_2X(v)}$, viewed as a section of the quotient  line bundle $T(\d X)/T(\d_2X(v))$ over $\d_2X(v)$, is transversal to its zero section.

For generic fields, this structure replicates itself: the cuspidal locus $\d_3X(v)$ is defined as the locus where $v$ is tangent to $\d_2X(v)$; $\d_3X(v)$ is divided into two manifolds, $\d_3^+X(v)$ and $\d_3^-X(v)$. In  $\d_3^+X(v)$, the field is directed inward of $\d_2^+X(v)$ or is tangent to its boundary, in  $\d_3^-X(v)$, outward of $\d_2^+X(v)$ or is tangent to its boundary. We can repeat this construction until we reach the zero-dimensional stratum $\d_{n+1}X(v) = \d_{n+1}^+X(v) \cup  \d_{n+1}^-X(v)$.  

To achieve some uniformity in the notations, put $\d_0^+X =_{\mathsf{def}} X$ and $\d_1X =_{\mathsf{def}} \d X$. 

Thus a  generic vector field $v$ on $X$ should give rise to two  stratifications: 
\begin{eqnarray}\label{eq2.0}
\d X =_{\mathsf{def}} \d_1X \supset \d_2X(v) \supset \dots \supset \d_{n +1}X(v), \nonumber \\ 
X =_{\mathsf{def}} \d_0^+ X \supset \d_1^+X(v) \supset \d_2^+X(v) \supset \dots \supset \d_{n +1}^+X(v), 
\end{eqnarray}
the first one by closed submanifolds, the second one---by compact ones.  Here $\dim(\d_jX(v)) = \dim(\d_j^+X(v)) = n +1 - j$. \smallskip

We will use often the notation ``$\d_j^\pm X$" instead of ``$\d_j^\pm X(v)$" when the vector field $v$ is fixed or its choice is obvious. \smallskip

These considerations motivate a more formal

\begin{definition}\label{boundary_generic} 
Let $X$ be a compact smooth $(n+1)$-dimensional manifold with boundary $\d X \neq \emptyset$, and $v$ a smooth vector field on $X$.  

We say that  $v$  is {\sf boundary generic} if the vector field $v|_{\d X}$ does not vanish and produces a filtrations of $X$ as in (\ref{eq2.0}). Its strata $\{\d_j^+X \subset \d_jX\}_{1 \leq j \leq n+1}$  are defined inductively in $j$ as follows:

\begin{itemize}
\item $\d_0X =_{\mathsf{def}} \d X$, $\d_1X =_{\mathsf{def}} \d X$ \footnote{So $\d_0X$ and $\d_1X$---the base of induction---do not depend on $v$.},
\item $v$, viewed as a section of the tangent bundle $T(X)$, is transversal to its zero section,
\item  for each $k \in [1, j]$, the $v$-generated stratum $\d_kX$ is a closed smooth submanifold of  $\d_{k-1}X$,
\item  the field  $v$, viewed as section of the quotient 1-bundle  $$T_k^\nu =_{\mathsf{def}} T(\d_{k-1}X)/ T(\d_kX) \to \d_kX,$$ is transversal to the zero section of $T_k^\nu \to \d_kX$ for all $k \leq j$. 
\item the stratum $\d_{j+1}X$ is the zero set of the section $v \in T_j^\nu$. 
\item the stratum $\d^+_{j+1}X \subset \d_{j+1}X$ is the locus where $v$ points inside of $\d_j^+X$.
\end{itemize}

We denote the space of boundary generic vector fields on $X$ by the symbol $\mathcal B^\dagger(X)$. \hfill $\diamondsuit$
\end{definition} 

By Theorem 3.4 from \cite{K2} (see also the second bullet of Theorem  6.6 from \cite{K7}), the smooth topological type of the stratification $\{\d_jX(v)\}_j$ is stable under perturbations of $v$ within the space $\mathcal B^\dagger(X)$ of boundary generic fields. The same argument shows that $\{\d_j^+X(v)\}_j$ is stable as well. 

\begin{definition}\label{convex}
We say that a boundary generic vector field $v$ is {\sf convex} if $\d_2^+X(v) = \emptyset$. When  $\d_2^-X(v) = \emptyset$, we say that the vector field $v$ {\sf concave}. 
\hfill $\diamondsuit$
\end{definition}

Note that convexity or concavity of $v$ implies that the locus $\d_3X(v) = \emptyset$. 
 \smallskip 
 
For the rest of the paper, we assume that the field $v$ on $X$ extends to a non-vanishing field $\hat v$ on some open manifold $\hat X$ which properly contains $X$ (see Fig. 6). We treat the extension $(\hat X, \hat v)$ as a {\sf germ} that contains $(X, v)$. One may think of $\hat X$ as being obtained from $X$ by attaching an external collar to $X$ along $\d_1X$. In fact, the treatment of $(X, v)$ will not depend on the germ of extension $(\hat X, \hat v)$, but many constructions are simplified by introducing an extension.
\smallskip

The trajectories $\g$ of a boundary generic vector field $v$ on $X$ interact with the boundary $\d X$ so that each point $a \in \g \cap \d X$ acquires a {\sf multiplicity} $m(a) \in \N$, the order of tangency of $\g$ to $\d X$ at $a$. We associate a {\sf divisor} $$D_\g = \sum_{a \in \g \cap \d X} m(a)\cdot a$$ with each $v$-trajectory $\g$. In fact, for any boundary generic $v$, $m(a) \leq \dim(X)$ and the support of $D_\g$ is finite (\cite{K2}).  
\smallskip

So we associate also a finite ordered sequence $\om(\g) = (\om_1, \om_2, \dots, \om_q)$ of {\sf multiplicities} with each $v$-trajectory $\g$. The multiplicity $\om_i$ is the order of tangency between the curve $\g$ and the hypersurface $\d X$ at the $i^{th}$ point of the finite set $\g \cap \d X$. The linear order in $\g \cap \d X$ is determined by $v$.   

Such sequences form a {\sf poset} $(\mathbf\Om, \succ)$, the partial order ``$\succ$" in $\mathbf\Om$ is defined in terms of two types of elementary operations: {\sf merges} $\{\mathsf M_i\}_i$ and {\sf inserts} $\{\mathsf I_i\}_i$ 
The operation $\mathsf M_i$ merges a pair of adjacent entries $\om_i, \om_{i+1}$ of $\om = (\om_1, \dots, \om_i, \om_{i+1}, \dots, \om_q)$ into a single component $\tilde\om_i = \om_i + \om_{i+1}$, thus forming a new shorter sequence $\mathsf M_i(\om) = (\om_1, \dots, \tilde\om_i, \dots, \om_q)$. The operation $\mathsf I_i$ either insert $2$ in-between $\om_i$ and $\om_{i+1}$, thus forming a new longer sequence $\mathsf I_i(\om) = (\dots, \om_i, 2, \om_{i+1}, \dots)$, or, in the case of $\mathsf I_0$, appends $2$ before the sequence $\om$, or,  in the case $\mathsf I_q$,  appends $2$ after the sequence $\om$. 

 So the {\sf merge} operation 
$\mathsf M_j: \mathbf\Om\to \mathbf\Om$ \;  
sends $\om = (\om_1,\ldots , \om_\ell)$ to the composition $$\mathsf M_j(\om) = (M_j(\om)_1,\ldots, M_j(\om)_{\ell-1}),$$ 
where, for any $j \geq \ell$, one has $\mathsf M_j(\om) = \om$,  and for $1 \leq j < \ell$, one has 
\begin{eqnarray}\label{eq2.M} 
\mathsf M_j(\omega)_i & = & \omega_i \; \textrm{ if }\, i < j,\\ \nonumber
\mathsf M_j(\omega)_j  & = & \omega_j + \omega_{j+1},\\ \nonumber
\mathsf M_j(\omega)_i  & = & \omega_{i + 1} \; \textrm{ if }\, i+1 < j \leq  \ell-1.
\end{eqnarray}

Similarly, we introduce the {\sf insert} operation 
$\mathsf I_j: \mathbf \Om \to\mathbf \Om$ that 
sends $\om = (\om_1,\ldots , \om_\ell)$ to the composition $\mathsf I_j(\om) = (I_j(\om)_1,
\ldots, I_j(\om)_{\ell+1}),$ where for any $j > \ell+1$,  one has $\mathsf I_j(\om) = \om$, and  for $1 \leq j \leq \ell+1$, one has
\begin{eqnarray}\label{eq2.I}  
\mathsf I_j(\omega)_i & = & \omega_i\; \textrm{ if }\, i < j, \\ \nonumber
\mathsf I_j(\omega)_j & = & 2,\\ \nonumber
\mathsf I_j(\omega)_i & = & \omega_{i - 1} \; \textrm{ if }\, j \leq i \leq \ell+1. 
\end{eqnarray}


We define $\om \succ \om'$ if one can produce $\om'$ from $\om$ by applying a sequence of these elementary operations.  
\smallskip

For each trajectory $\g$ of a boundary generic and traversing $v$, we introduce two important quantities:
\begin{eqnarray}\label{eq2.1}
m(\g) =_{\mathsf{def}} \sum_{a \in \g \cap \d_1X }\; m(a), \quad \text{and} \quad m'(\g) =_{\mathsf{def}} \sum_{a \in \g \cap \d_1X }\; (m(a) -1),
\end{eqnarray} 
the {\sf multiplicity} and the {\sf reduced multiplicity}.
 
Similarly, for a sequence $\omega = (\omega_1, \omega_2, \, \dots \, ,  \omega_q)$, we introduce the {\sf norm} and the {\sf reduced norm} of $\omega$ by the formulas: 

\begin{eqnarray}\label{eq2.2}
|\omega| =_{\mathsf{def}} \sum_i\; \omega_i \quad \text{and} \quad |\omega|'  =_{\mathsf{def}} \sum_i\; (\omega_i -1).
\end{eqnarray}
Note that $q$, the cardinality of the {\sf support} of $\omega$, is equal to $|\omega| - |\omega|'$.
\smallskip

For boundary generic and traversing vector fields $v$, the trajectory space $\mathcal T(v)$ is stratified by subspaces, labeled by the elements $\om = (\om_1, \dots , \om_q)$ of an universal poset $\mathbf\Om^\bullet$. Its elements form a subset of $\mathbf\Om$, but not a sub-poset (see \cite{K3} for the accurate definition of the partial order $\succ_\bullet$ in $\mathbf\Om^\bullet$). For $q > 1$, the first and the last entries of $\om \in \mathbf\Om^\bullet$ are \emph{odd} positive integers, the rest are \emph{even}.  When $q =1$, $\om=(\om_q)$ must be even. For a boundary generic $v$, each $\om_i \leq \dim(X)$.  
\smallskip
\smallskip

In this paper, we consider also an important subclass of traversing and boundary generic fields, which we call {\sf traversally  generic} (see Definition \ref{traversally_generic} below or Definition 3.2 from \cite{K2}). Such fields admit special flow-adjusted coordinate systems, in which the boundary is given by quite special polynomial equations (see formula (\ref{eq_dX})) and the trajectories are parallel to the preferred coordinate axis (see  \cite{K2}, Lemma 3.4). \smallskip

Given a boundary generic and traversing vector field $v$, for each trajectory $\g$, consider the finite set $\g \cap \d_1X = \{a_i\}_i$ and the collection of tangent spaces $\{T_{a_i}(\d_{j_i}X^\circ)\}_i$ to the pure strata $\{\d_{j_i}X^\circ\}_i$. Each space $T_{a_i}(\d_{j_i}X^\circ)$ is transversal to the curve $\g$.  
 \smallskip 

Let $S$ be a local transversal section of the $\hat v$-flow at a point $a_\star \in \g$, and let $\mathsf T_\star$ be the  space tangent to $S$ at $a_\star$. Each space $T_{a_i}(\d_jX^\circ)$, with the help of the $\hat v$-flow, determines a vector subspace $\mathsf T_i = \mathsf T_i(\g)$ of  $\mathsf T_\star$. It is the image of  the tangent space $T_{a_i}(\d_jX^\circ)$ under the composition of two maps: 

(1) the differential of the $v$-flow-generated diffeomorphism that maps $a_i$ to $a_\star$, and

(2) the linear  projection $T_{a_{\star}}(X) \to \mathsf T_\star$, whose kernel is generated by $v(a_\star)$. 
\smallskip

The configuration $\{\mathsf T_i\}$ of {\sf affine} subspaces  $\mathsf T_i \subset \mathsf T_\star$ is called {\sf generic} (or {\sf stable}) when all the multiple intersections of spaces from the configuration have the least possible dimensions, consistent with the dimensions of $\{\mathsf T_i\}$. In other words, $$\textup{codim}(\bigcap_{s} \mathsf T_{i_s},  \mathsf T_\star) = \sum_s \textup{codim}(\mathsf T_{i_s},  \mathsf T_\star)$$ for any subcollection $\{\mathsf T_{i_s}\}$ of spaces from the list $\{\mathsf T_i\}$.

Consider the case when $\{\mathsf T_i\}$ are \emph{vector} subspaces of $\mathsf T_\star$.  If we interpret each $\mathsf T_i$ as the kernel of a linear epimorphism $\Phi_i:  \mathsf T_\star \to \R^{n_i}$, then the property of  $\{\mathsf T_i\}$ being generic can be reformulated as the property of  the direct product map $\prod_i \Phi_i:  \mathsf T_\star \to \prod_i  \R^{n_i}$ being an epimorphism.  In particular, for a generic configuration of affine subspaces, if a point  belongs to several  $\mathsf T_i$'s, then the sum of their codimensions $n_i$ does not exceed the dimension of the ambient space $\mathsf T_\star$. 
\smallskip

The definition below resembles and is inspired by the ``{\sf Condition NC}" imposed on, so called, {\sf Boardman maps} between smooth manifolds (see \cite{GG}, page 157, for the relevant definitions). In fact, for generic traversing vector fields $v$, the $v$-flow delivers germs of Boardman maps $p(v, \g): \d_1X \to \R^n$, available in the vicinity of each trajectory $\g$. Here $\R^n$ is identified with a transversal section of the flow in the vicinity  of $\g$.

\begin{definition}\label{traversally_generic} A traversing field $v$ on $X$ is called {\sf traversally generic} if: 
\begin{itemize}
\item  the field is boundary generic in the sense of Definition \ref{boundary_generic},
\item for each $v$-trajectory $\g \subset X$ (not a singleton), the collection of  subspaces $\{\mathsf T_i(\g)\}_i$  is generic in $\mathsf T_\star$: that is, the obvious quotient map $\mathsf T_\star \to \prod_i \big(\mathsf T_\star/ \mathsf T_i(\g)\big)$ is surjective.   
\end{itemize}

We denote by $\mathcal V^\ddagger(X)$ the space of all traversally generic fields on $X$. \hfill $\diamondsuit$
\end{definition}


\begin{remark}\label{rem2.1} 
In particular, the second bullet in Definition \ref{traversally_generic} implies the inequality $$\sum_i \textup{codim}(\mathsf T_i(\g), \mathsf T_\star) \leq \dim(\mathsf T_\star) = n.$$   In other words, for traversally generic fields, the reduced multiplicity of each trajectory $\g$ satisfies the inequality
\begin{eqnarray}\label{eq3.4}
m'(\g) = \sum_i (j_i - 1) \leq n.
\end{eqnarray}

Evidently, the property of the configuration $\{\mathsf T_i(\g)\}_i$ being generic in $\mathsf T_\star$ does not depend on the choice of the point $a_\star \in \g$ and the smooth transversal flow section $S$ at $a_\star$.
\end{remark}

So all sufficiently close (in the $C^\infty$-topology) vector fields to a traversally generic field will remain traversally generic. Moreover, by Theorem 3.5 from \cite{K2}, the space $\mathcal V^\ddagger(X)$ is open and \emph{dense} in $\mathcal V_{\mathsf{trav}}(X)$. This property of $\mathcal V^\ddagger(X)$ will be of great importance for our endeavor. 
\smallskip

For traversally  generic vector fields $v$, the trajectory space $\mathcal T(v)$ is stratified by subspaces, labeled by the elements $\omega$ of another \emph{universal subposet} $\mathbf \Om^\bullet_{'\langle n]} \subset \mathbf \Om^\bullet$, defined by the constraint $|\om|' \leq n$. It depends only on $\dim(X) = n+1$ (see \cite{K3} for the definition and properties of $\mathbf \Om^\bullet_{'\langle n]}$). 
\smallskip

Let us revisit the stratum $\d_jX =_{\mathsf{def}} \d_jX(v)$, the locus of points $a \in \d_1X$ such that the multiplicity of the $v$-trajectory $\g_a$ through $a$ at $a$ is greater than or equal to $j$. This locus has an alternative description in terms of  an auxiliary smooth function $z: \hat X \to \R$ that satisfies the following three properties:
\begin{eqnarray}\label{eq2.3}
\end{eqnarray}
\begin{itemize}
\item $0$ is a regular value of $z$,   
\item $z^{-1}(0) = \d X$, 
\item $z^{-1}((-\infty, 0]) = X$. 
\end{itemize}

In terms of $z$, the locus $\d_jX$ is defined by the equations: 
$$\{z =0,\; \mathcal L_vz = 0,\; \dots, \;  \mathcal L_v^{(j-1)}z = 0\},$$
where $\mathcal L_v^{(k)}$ stands for the $k$-th iteration of the Lie derivative operator $\mathcal L_v$ in the direction of $v$ (see \cite{K2}). 

The pure stratum $\d_jX^\circ \subset \d_jX$ is defined by the additional constraint  $\mathcal L_v^{(j)}z \neq 0$. The locus $\d_jX$ is the union of two loci: 

$\mathbf {(1)}$ $\d_j^+X$, defined by the constraint  $\mathcal L_v^{(j)}z \geq  0$, and 

$\mathbf{(2)}$ $\d_j^-X$, defined by the constraint  $\mathcal L_v^{(j)}z \leq  0$. 

The two loci, $\d_j^+X$ and $\d_j^-X$, share the common boundary $\d_{j+1}X$.
\smallskip

The following  lemma is on the level of definitions.

\begin{lemma}\label{boundary_generic_A}  A vector field $v$ on a smooth $(n+1)$-manifold $X$ with boundary is boundary generic if and only if, for each $j \in [1, n+1]$, the differential $j$-form
\begin{eqnarray}\label{Xi_j} 
\Xi_j(z, v) := dz \wedge \mathcal L_v(dz) \wedge \ldots (\mathcal L_v)^{j-1}(dz)
\end{eqnarray}
does not vanish along the locus $\d_jX(v)$. \hfill $\diamondsuit$
\end{lemma}

The next lemma may be found in \cite{Mor} or in \cite{K2}. 

\begin{lemma}\label{local_boundary_generic} 
 Let $v$ be a boundary generic vector field on a $(n+1)$-dimensional smooth manifold $X$ with boundary. Let a $v$-trajectory $\g_\star$ be tangent to $\d_1X$ at a point $b \in \g_\star \cap \d_1X$ with the order of tangency $j \in [1, n+1]$. 
 
In the vicinity of $b$ in $X$, there exists a system of smooth coordinates 
$\{u, \vec x, \vec y\} := \{u, x_0, \dots , x_{j-2}, y_1, \dots y_{n-j+1}\}$ such that:
\begin{itemize}
\item the boundary $\d_1X$ is given by the equation
\begin{eqnarray}\label{eqBG}
P(u, \vec x) := u^j +  \sum_{\ell=0}^{j-2} x_\ell\, u^\ell = 0,
\end{eqnarray}
and $X$ by the inequality $P(u, \vec x) \leq 0$,
\item each $v$-trajectory is given by freezing the coordinates $\{\vec x, \vec y\},$ subject to the constraint $P(u, \vec x) \leq 0$. \hfill $\diamondsuit$
\end{itemize}
\end{lemma}

Lemma \ref{local_boundary_generic} implies the next lemma (see \cite{K2}, Lemma 3.4, or \cite{K7}, Lemma 6.4, for its validation).

\begin{lemma}\label{lem2.A} Let $X$ be a $(n+1)$-dimensional compact connected smooth manifold $X$ with boundary and $v$ a traversing boundary generic vector field on $X$. Let $\g$ be a $v$-trajectory of a combinatorial type $\om$. Then there is a $\hat v$-adjusted neighborhood $U \subset \hat X$ of $\g$ and a system of coordinates $(u, \vec x): U \to \R \times \R^n$ such that $U$ is given by the inequalities $P(u, \vec  x) \leq 0$, $\|\vec x\| < \e$, where $$P(u, \vec x) := u^{|\om|} + \sum_{\ell =0}^{|\om|-1} \phi_\ell(\vec x)\, u^\ell$$ 
and $\{\phi_\ell(\vec x)\}_\ell$ are smooth functions. The real divisor of $P(u, \vec 0)$ has the combinatorial type $\om$. Each $\hat v$-trajectory in $U$ is given by freezing the coordinate $\vec x \in \R^n$, subject to the constraint $P(u, \vec x) \leq 0$.
\hfill $\diamondsuit$
\end{lemma}

Let $v$ be a traversing, boundary generic vector field. For each $v$-trajectory $\g$ and each point $a_i \in \g \cap \d X$ of multiplicity $j_i := j(a_i)$, we consider the form $\Xi_j(z, v)|_{a_i} \in \bigwedge^{j_i} T^\ast_{a_i}X$ (see (\ref{Xi_j})) and spread it via the $v$-flow along $\g$. We denote by $\tilde\Xi_{j_i}(z, \g)$ the resulting section ($j_i$-form) of the bundle $\bigwedge^{j_i} T^\ast X|_\g$. Lemma \ref{local_boundary_generic} admits the following interpretation.

\begin{lemma}\label{traversally_generic_A}  A traversing and boundary generic vector field $v$ on a smooth $(n+1)$-manifold $X$  with boundary is traversally generic if and only if, for each trajectory $\g$, the $m(\g)$-dimensional differential form
$$\tilde\Xi(z, \g) =_{\mathsf{def}} \bigwedge_{i=1}^s \tilde\Xi_{j_i}(z, \g) \in \bigwedge^{|\om_\g|}T^\ast X \big |_\g$$ 
(where $s = \#(\g \cap \d X)$ and $|\om_\g| = \sum_{i=1}^s j_i$) does not vanish along $\g$.
\hfill $\diamondsuit$
\end{lemma}

For a \emph{traversally generic} $v$ (see Definition \ref{traversally_generic}) on a $(n+1)$-dimensional $X$, the vicinity $U \subset \hat X$ of each $v$-trajectory $\g$ of a combinatorial type $\omega \in \mathbf\Om^\bullet$ has a special coordinate system $(u, x, y): U \to \R\times \R^{|\omega|'} \times \R^{n-|\omega|'}.$ By by Lemma 3.4 from \cite{K2} (see also Lemma 6.4 in \cite{K7}), in these coordinates, the boundary $\d_1X := \d X$ is given by the polynomial equation
\begin{eqnarray}\label{eqVERSAL}
 P(u, x) := \prod_i \big[(u-\a_i)^{\omega_i} + \sum_{\ell = 0}^{\omega_i-2} x_{i, \ell}(u -\a_i)^\ell \big] = 0
 \end{eqnarray}
 of an even degree $|\omega|$ in $u$. Here $x =_{\mathsf{def}} \{ x_{i, l}\}_{i,\ell}$,  and the numbers $\{\a_i\}_i$ are all distinct real roots of the polynomial $P(u, 0)$,  ordered so that $\a_i < \a_{i+1}$ for all $i$.  

At the same time, $X$ is given by the polynomial inequality $\{P(u, x) \leq 0\}$.  Each $v$-trajectory in $U$ is produced by freezing all the coordinates $x, y$, while letting $u$ to be free. Formula (\ref{eqVERSAL}) should be compared with formula (\ref{eqBG}).

In fact, by choosing $\a_i = i$, we may rewrite this equation for $\d X$ in $U$ as 
\begin{eqnarray}\label{eq_dX}
 \wp_\om(u, x) := \prod_i \big[(u-i)^{\omega_i} + \sum_{\ell = 0}^{\omega_i-2} x_{i, \ell}(u -i)^\ell \big] = 0
 \end{eqnarray}
 (where $|\om|' \leq \dim X -1$, $|\om| \leq 2\cdot \dim X$, and $|\om| \equiv 0 \mod 2$.
 That equation may be viewed as the working definition of a traversally generic vector field.


\section{On the trajectory spaces for traversally generic flows}

Let $v$ be a traversing vector field. By collapsing each $v$-trajectory to a singleton, we produce the trajectory space $\mathcal T(v)$, equipped with the quotient topology.\smallskip

We denote by $X(v, \omega)$ the union of $v$-trajectories whose patterns of tangency to $\d_1 X := \d X$ are of a given combinatorial type $\omega \in \mathbf \Omega^\bullet$. We use the notation $X(v, \omega_{\succeq})$ for its closure $\cup_{\omega' \preceq \omega}\; X(v, \omega')$.
\smallskip

For a traversally generic $v$, each pure stratum $\mathcal T(v, \omega) \subset \mathcal T(v)$ is an open smooth manifold, and as such has a ``conventional" tangent bundle. In particular, the pure strata of maximal dimension $n$ have tangent bundles. It turns out that these ``honest" tangent $n$-bundles extend across the singularities of the space $\mathcal T(v)$ to form a $n$-bundle $\tau(\mathcal T(v))$ over $\mathcal T(v)$! However, at the singularities, no  exponential map (that takes a vector from  $\tau(\mathcal T(v))$ to a point in $\mathcal T(v)$) is available---the surrogate tangent bundle $\tau(\mathcal T(v))$ does not reflect faithfully the local geometry of the trajectory space $\mathcal T(v)$.

In order to define the dual of the bundle $\tau(\mathcal T(v))$ intrinsically, we need to consider a surrogate of  smooth structure on the singular space $\mathcal T(v)$. 

\begin{definition}\label{def2.1} Let $v$ be a smooth traversing vector field on a smooth compact and connected manifold $X$.  Let $\Gamma: X \to \mathcal T(v)$ be the projection that takes each point $x \in X$ to the trajectory $\g_x \in \mathcal T(v)$ that contains $x$. 

We say that a function $h: \mathcal T(v) \to \R$ is {\sf smooth}, if the composition $h \circ \Gamma$ is smooth on $X$. 

We denote by $C^\infty(\mathcal T(v))$ the algebra of all smooth functions on the space $\mathcal T(v)$. \hfill 
$\diamondsuit$
\end{definition}

\begin{definition}\label{def2.2} Let $v_1, v_2$ be two traversing  vector fields on manifolds $X_1, X_2$, respectively.
\begin{itemize}
\item A map $\Phi: \mathcal T(v_1) \to \mathcal T(v_2)$ is called {\sf smooth}, if for any function $h$ from $C^\infty(\mathcal T(v_2))$, its pull-back $\Phi^\ast(h) \in C^\infty(\mathcal T(v_1))$. 
\item A bijective map $\Phi: \mathcal T(v_1) \to \mathcal T(v_2)$ is called a {\sf smooth diffeomorphism}, is both $\Phi$ and $\Phi^{-1}$ are smooth.   \hfill $\diamondsuit$
\end{itemize} 
\end{definition}

For any traversing field $v$,  the algebra $C^\infty(\mathcal T(v))$ of smooth functions on the trajectory space $\mathcal T(v)$ can be identified with the subalgebra of $C^\infty(X)$, formed by functions $f: X \to \R$ with the property $\{\mathcal L_v(f) = df(v) = 0\}$, where $\mathcal L_v$ stands for the $v$-directional derivative\footnote{This property does not depend on an extension $(\hat X, \hat v)$ of $(X, v)$.}. Such functions are constant along each trajectory $\g \subset X$.

We denote by $C^\infty_\g(\mathcal T(v))$ the algebra of germs of smooth functions from $C^\infty(\mathcal T(v))$ at a given point $\g \in \mathcal T(v)$. Let  $\mathsf m_\g(\mathcal T(v)) \lhd C^\infty_\g(\mathcal T(v))$ be the maximal ideal, formed by the the germs  that vanish at $\g$, and let $\mathsf m_\g^2(\mathcal T(v))$ be the square of the ideal $\mathsf m_\g(\mathcal T(v))$. 

Then the quotients $\mathsf m_\g(\mathcal T(v))/ \mathsf m_\g^2(\mathcal T(v))$ are real $n$-dimensional vector spaces. Indeed, since the pull-back of smooth functions on $\mathcal T(v)$  are the smooth functions on $X$ that are constants along each trajectory $\g$, the quotient  $\mathsf m_\g(\mathcal T(v))/ \mathsf m_\g^2(\mathcal T(v))$  can be canonically identified  with the quotient $\mathsf m_x(S)/ \mathsf m^2_x(S)$.  Here $S$ is a germ of a smooth transversal section of the $\hat v$-flow at $x = \g \cap S$, and $\mathsf m_x(S)$ denotes the maximal ideal in the algebra $C^\infty(S)$, an ideal comprised of functions that vanish at $x$.  It is well-known that $\mathsf m_x(S)/\mathsf m^2_x(S)$ can be canonically identified with the cotangent space $T_x^\ast(S)$ via the correspondence $f \Rightarrow df$, where the germ of $f: S \to \R$ at $x$ belongs to the ideal $\mathsf m_x(S)$. Therefore the spaces $$\tau_\g^\ast(\mathcal T(v)) =_{\mathsf{def}} \mathsf m_\g(\mathcal T(v))/ \mathsf m_\g^2(\mathcal T(v))$$ form a vector $n$-bundle $\tau^\ast(\mathcal T(v))$ over $\mathcal T(v)$. It is dual to $\tau(\mathcal T(v))$ under the construction. The pull-back 
$\Gamma^\ast\big(\tau^\ast(\mathcal T(v))\big)$  can be identified with the subbundle $\tau^\ast(v)$ of the cotangent bundle $T^\ast(X)$, formed by the ``horizontal" $1$-forms $\a$ such that $\a(v) = 0$ and $\mathcal L_v(\a) = 0$. The identification is via  the correspondence $\Gamma^\ast(f) \Rightarrow d(\Gamma^\ast(f))$, where $f \in \mathsf m_\g(\mathcal T(v))$.

Now we \emph{define} $\tau(\mathcal T(v))$ as the dual bundle of $\tau^\ast(\mathcal T(v))$. 
\smallskip

Let $(11) \in \mathbf \Omega^\bullet_{'\langle n]}$ denote the unique maximal element of the poset; it labels the trajectories that intersect the boundary $\d X$ only at a pair of distinct points, where they are transversal to the boundary.  

\begin{lemma}\label{lem2.1} For any traversing field $v$, the tangent bundles to the components of the maximal stratum $\mathcal T(v, (11))$ extend to a $n$-dimensional vector bundle $\tau(\mathcal T(v))$ over the trajectory space $\mathcal T(v)$. 

Moreover, for a traversally  generic field $v$ and each element $\omega \in \mathbf \Omega^\bullet_{'\langle n]}$, the tangent bundle of the pure stratum $\mathcal T(v, \omega)$ embeds in $\tau(\mathcal T(v))|_{\mathcal T(v, \omega)}$ as a subbundle with a canonically trivialized complement. 
\end{lemma}

\begin{proof} We already have observed that the pull-back  $\Gamma^\ast(\tau^\ast(\mathcal T(v)))$ of the cotangent bundle $\tau^\ast(\mathcal T(v))$ can be identified with the bundle $\tau^\ast(v)$ of the flow-invariant $1$-forms on $X$ that vanish on $v$.

The map $\Gamma:  X(v, (11)) \to \mathcal T(v, (11))$ is a fibration with a closed segment for the fiber. Therefore $\Gamma$ admits a smooth section $S_{(11)} \subset X(v, (11))$ which is transversal to the $v$-trajectories.  Consider a decomposition of the $(n+1)$-bundle $T(X)|_{S_{(11)}}$ into the tangent $n$-bundle $T(S_{(11)})$ and a line bundle $L$ tangent to the $v$-trajectories. With the help of this decomposition, the cotangent bundle $T^\ast(S_{(11)})$ can be identified with the restriction $\tau^\ast(v)|_{S_{(11)}}$  of  $\tau^\ast(v)$ to $S_{(11)}$.
Using the isomorphism  $\tau^\ast(v)|_{S_{(11)}} \approx \Gamma^\ast(\tau^\ast(\mathcal T(v)))|_{S_{(11)}}$, we identify the cotangent  bundle $T^\ast(S_{(11)})$ with the bundle $\tau^\ast(\mathcal T(v))|_{S_{(11)}}$, a bundle that evidently is defined on the whole space $\mathcal T(v)$.
\smallskip

A similar conclusion holds for any traversally generic vector field $v$ \footnote{Here perhaps a much weaker assumption about $v$ will do.} and each $\omega \in \Omega^\bullet_{'\langle n]}$: by Lemma 3.4 from \cite{K2}, the map $\Gamma:  X(v, \omega) \to \mathcal T(v, \omega)$ is a fibration with its base being an open smooth $(n - |\omega|')$-manifold and with a closed segment for the fiber, the fiber being consistently oriented by $v$. Therefore $\Gamma$ admits a smooth section $S_\omega$. The cotangent bundle $\tau^\ast(S_\omega)$ can be identified with the cotangent bundle $\tau^\ast(\mathcal T(v, \omega))|_{\mathcal T(v, \omega)}$, a bundle that embeds  into the bundle $\tau^\ast(\mathcal T(v))$. 

So the only non-trivial statement of the lemma is the existence of a preferred trivialization in the quotient bundle $\tau(\mathcal T(v))|_{\mathcal T(v, \omega)}\,/ \tau(\mathcal T(v, \omega))$.  It follows from the last claim of Theorem \ref{th2.1} below.  Thus  $\Psi: \tau(\mathcal T(v, \omega)) \oplus \underline \R^{|\omega|'}\; \approx \; \tau(\mathcal T(v))|_{\mathcal T(v, \omega)},$
where the bundle isomorphism $\Psi$ is canonically defined by $v$. 
\end{proof}
\smallskip

\begin{corollary}\label{cor2.1} For a traversing vector field $v$ on $X$, the stable characteristic classes of the tangent bundles $\tau(\mathcal T(v))$ and $\tau(X)$ coincide via the cohomological isomorphism induced by the projection $\Gamma: X \to \mathcal T(v)$.
\end{corollary}

\begin{proof} Note that  $T(X) \approx \Gamma^\ast(\tau(\mathcal T(v))) \oplus \underline{\R}$. Therefore,  the cohomological isomorphism induced by $\Gamma$ (see Theorem 5.1, \cite{K3}) helps to identify the stable characteristic classes of  $\tau(\mathcal T(v))$ and $T(X)$.
\end{proof}

For a traversally generic $v$, the space  $\mathcal T(v)$ comes equipped with two distinct \emph{intrinsically-defined orientations} of its pure strata $\{\mathcal T(v, \omega)\}_\omega$. These orientations depend only on $v$ and the preferred orientation of $X$.

\begin{theorem}\label{th2.1} Let $X$ be a smooth oriented compact $(n+1)$-manifold, and $v$ a traversally generic vector field. Then 
\begin{itemize}
\item each component of any pure stratum $\mathcal T(v,\omega)$, where $\omega \in \mathbf\Omega^\bullet_{'\langle n]}$ and $|\omega|' > 0$,  acquires two distinct orientations, called  {\sf preferred} and {\sf versal}. Switching the orientation of $X$ affects both orientations of $\mathcal T(v,\omega)$  by the same factor $(-1)^{|sup(\omega)|}$. \smallskip

\item With the help of these two orientations, each  component of  $\mathcal T(v,\omega)$ acquires one of the  two polarities `` $\oplus$" and `` $\ominus$". They do not depend on the orientation of $X$.
\smallskip

\item Each manifold $X(v, \omega)$ comes equipped with a $v$-induced normal framing in $X$. Similarly, the normal $|\omega|'$-dimensional bundle $$\nu(\mathcal T(v, \omega)) =_{\mathsf{def}} \tau(\mathcal T(v))|_{\mathcal T(v, \omega)}/ \tau(\mathcal T(v, \omega))$$ acquires a $v$-induced preferred  framing. 
\end{itemize}
\end{theorem}

\begin{proof} We extend the field $v$ on $X$ to a non-vanishing field $\hat v$ on $\hat X \supset X$. 
Local transversal sections $S$ of the $\hat v$-flow have a well-defined orientation due to the global orientation of $X$ and the preferred orientation of the $v$-trajectories. \smallskip

For a traversally generic $v$ on a $(n+1)$-dimensional $X$,  each $v$-trajectory $\g$ of the combinatorial type $\omega$  has  a flow adjusted neighborhood $U \subset \hat X$, equipped with a special coordinate system $(u, x, y): U \to \R\times \R^{|\omega|'} \times \R^{n-|\omega|'}.$ By Lemma 3.4  and formula $(3.17)$ from \cite{K2},  the boundary $\d X$ is given in these coordinates by the polynomial equation  $\{P(u, x) =0\}$
 in $u$ of an even degree $|\omega|$ (see (\ref{eqVERSAL})). Here $x =_{\mathsf{def}} \{ x_{i, \ell}\}_{I,\ell}$,  and the numbers $\{\a_i\}_i$ are the distinct real roots of the polynomial $P(u, 0)$,  ordered so that $\a_i < \a_{i+1}$ for all $i$.  
At the same time, $X$ is given by the polynomial inequality $\{P(u, x) \leq 0\}$.  Each $v$-trajectory in $U$ is produced by freezing all the coordinates $x, y$, while letting $u$ to be free.

We order  the coordinates $\{x_{i,\ell}\}_{i, \ell}$ lexicographically: first we order them by the increasing $i$'s; then, for a fixed $i$, the ordering among $\{x_{i, \ell}\}_\ell$ is defined by the increasing powers $\ell$ of the binomial $(u - \a_i)$ in the formula (\ref{eqVERSAL}). This ordering of $\{x_{i, \ell}\}_{i, \ell}$, together with the orientation in the flow section $S$ (induced with the help of $v$ by the orientation of $X$) gives rise to an orientation of the $y$-coordinates. They correspond to the space, tangent to the pure stratum $\mathcal T(v, \omega)$ at $\g$. 
\smallskip

We still have to check that this ordering of $\{x_{i, \ell}\}_\ell$ is determined by the \emph{geometry of tangency} and does not depend on a particular choice of the special coordinates $\{x_{i, \ell}\}_\ell$. 

Consider a $v$-trajectory $\g$. Let $\g \cap \d_1X = \coprod_i a_i$, a finite set of points. In the vicinity of $a_i \in \d_{j_i}X^\circ$, we write down the auxiliary function $z$ from (\ref{eq2.3})  in two ways:  $$\text{as}\;\; u^j + \sum_{l = 0}^{j - 2} \phi_\ell(x) u^\ell,\;\; \text{and as} \;\;  \big(u^j + \sum_{l = 0}^{j - 2} x_\ell u^\ell\big)\, Q(u,x).$$ 
Here,   $j =_{\mathsf{def}} j_i = \omega_i$, $x =_{\mathsf{def}}\{x_\ell = \{x_{i, \ell}\}_i\}_\ell$, $\phi_\ell(0) = 0$, and $q =_{\mathsf{def}} Q(0,0) \neq 0$. 
\smallskip

Consider the smooth map  $\Phi: \R^{j-1} \to  \R^{j-1}$, given by the functions $\phi_0, \dots, \phi_{j-2}$.\smallskip

We aim to show that, at the origin $(u, x) = (0, 0)$, the following two exterior $(j - 1)$-forms are equal:
\begin{eqnarray}\label{eq2.5}
d \phi_0 \wedge d \phi_1 \wedge \dots  \wedge d \phi_{ j- 2}\;|_{(0, 0)} = d x_0 \wedge d x_1 \wedge \dots \wedge d x_{j -2}\;|_{(0, 0)}.
\end{eqnarray}
Hence the Jacobian $\det(D\Phi) > 0$ ---the two orientations, induced by two coordinate systems $\{\phi_\ell\}_\ell$ and $\{x_\ell\}_\ell$ in the vicinity of $a_i$, do agree. \smallskip

The argument validating (\ref{eq2.5}) is similar to the one we have used in \cite{K2}, Lemma 3.3. 

First note that $q =_{\mathsf{def}} Q(0, 0) \neq 0$ must be 1: just plug $x = 0$ in the identity  
\begin{eqnarray}\label{eq2.6}
u^j + \sum_{\ell = 0}^{j - 2} \phi_\ell(x) u^\ell = \big(u^j + \sum_{\ell = 0}^{j - 2} x_\ell u^\ell\big) Q(u,x).
\end{eqnarray}

Let $a(u)$ be the row-vector $(u^{j -2}, \dots , u, 1)$ and $d \phi$ be the column-vector $(d\phi_{j-2}, \dots , d\phi_1, d\phi_0)$ of 1-forms. Then the differential of the identity (2.6), 
modulo the ideal $\langle u^{j -1}, x \rangle$,  generated by the functions $u^{j -1}$ and $x_0, \dots , x_{j-2}$,  can be written as $$a \ast d\phi = Q a \ast d x\;\; \mod\, \langle u^{j -1}, x \rangle, $$
where ``$\ast$" stands for the matrix multiplication. 

We apply partial derivatives $\frac{\d}{\d u}, \dots , \frac{\d^{j-2}}{\d u^{j-2}}$ to the identity above to get a new system of identities: 
$$\frac{\d^k}{\d u^k}(a) \ast d\phi = \frac{\d^k}{\d u^k} (Qa)\ast d x \;\; \mod\, \langle u^{j -1 -k}, x \rangle,$$
where $k = 0, 1, \dots , j - 2$. Now put $u = 0$ and use that $q =1$ to get the following triangular system of identities, modulo the ideal $\langle  x \rangle$ generated by $\{x_\ell\}_\ell$: 
\[
\begin{array}{l}
d\phi_0 =  d x_0 \;\; \mod\; \langle  x \rangle \\
d\phi_1 = d x_1 + b_{1,0}\, d x_0 \;\; \mod\; \langle  x \rangle \\
d\phi_2 = d x_2 + b_{2,0}\, d x_0 +  b_{2,1}\, d x_1\;\; \mod\; \langle  x \rangle \\
\dots \\
d\phi_{j-2} = d x_{j-2} + b_{j-2, 0}\, d x_0 +  b_{j-2, 1}\, d x_1 + \dots + b_{j-2, j-3}\, d x_{j-3} \;\; \mod\; \langle  x \rangle
\end{array}
\] 
Here $b_{s, t}$ denote some functional coefficients whose computation we leave to the reader.
Now (\ref{eq2.5}) follows by taking exterior products of the 1-forms on the RHS and LHS of the system above  and letting $x = 0$. \smallskip

Let $\theta_i =_{\mathsf{def}}  dx_{i,0} \wedge \dots \wedge dx_{i,j_i - 2}$ and let $\theta =_{\mathsf{def}} \wedge_i \, \theta_i$. Then $d u \wedge \theta$, together with the volume form in $X$, define the volume form in the $y$-coordinates. Therefore the orientation of the space $\tau_\g(\mathcal T(v, \omega))$, tangent to the pure stratum $\mathcal T(v, \omega)$ at its typical point $\g$ (this space can be  identified with the space spanned by the vectors $\d_{y_1}, \dots , \d_{y_{n - |\omega|'}}$), is determined intrinsically by the local geometry of the $v$-flow in the vicinity of $\g \subset \hat X$. Let us call this orientation of $\tau_\g(\mathcal T(v, \omega))$ {\sf versal}. \smallskip

On the other hand, each manifold $\d_j X$, $j > 0$,  comes equipped with its own {\sf preferred} orientation, which  depends only on the stratification $\{\d_k^+X(v)\}_k$  and on the preferred orientation of $X$. Here is the recipe for its construction:  the orientation of $X$, with the help of the inward normals, induces a preferred orientation of $\d X$, and thus of $\d_1^\pm X$. In turn, the inward normals to $\d_2X = \d(\d_1^+X)$ in $\d_1^+X$ produce a preferred orientation of $\d_2X$, and thus of $\d_2^\pm X$. And the process goes on: the preferred orientation of $\d_{j - 1}X$, with the help of the inward normal  to $\d_jX$  in  $\d_{j - 1}^+X$, determines a preferred orientation of $\d_jX$, and hence of $\d_j^\pm X$.

So, along each trajectory $\g$, every space $\mathsf T_i$, tangent to $\d_{j_i}X^\circ$ and transversal to $\g$ at the point $a_i \in \g \cap \d_1X$, is preferably oriented. For a traversally  generic $v$, the $\hat v$-flow propagates these spaces $\mathsf T_i$'s along $\g$ in such a way that they form complementary vector bundles over $\g$. We order them by the increasing values of $i$. This ordering, together with the preferred orientations of the $\mathsf T_i$'s (based on the orientations of $\d_{j_i}^+X$), generates a new  preferred orientation of the tangent space $\tau_\g(\mathcal T(v, \omega))$. This preferred orientation may agree or disagree with the versal orientation of the same space, produced with the help of special coordinates in the vicinity of $\g$; recall that the versal orientation is based on the increasing powers of $(z - \a_i)$'s, a feature of the special coordinates. In the first case, we attach the polarity ``$\oplus$" to $\g$, in the second case, the polarity of $\g$ is defined to be ``$\ominus$". 

Therefore not only the components of pure strata $\mathcal T(v, \omega)$ are canonically oriented open manifolds, but they also come in two flavors: ``$\oplus$" and ``$\ominus$"!
\smallskip

We will exhibit an ordered collection of $|\omega|'$ linearly independent and globally defined $1$-forms (as in \cite{K2}, formula (3.30)) that produces a framing of the quotient bundle $$\nu^\ast(\mathcal T(v, \omega)) := \tau^\ast(\mathcal T(v))|_{\mathcal T(v, \omega)}/ \tau^\ast(\mathcal T(v, \omega)),$$ the ``normal cotangent bundle"  of  $\mathcal T(v, \omega)$ in $\mathcal T(v)$.  Let us outline their construction.  

For any $\g \in \mathcal T(v, \omega)$ and any two points $a, x \in \g$, denote by $\phi_{a, x}$ the germ (taken in the vicinity of $\g \subset \hat X)$ of the unique $v$-flow-generated  diffeomorphism  that maps $x$ to $a$. 

Fix an auxiliary function $z: \hat X \to \R$ as in (\ref{eq2.3}). For each point $a_i \in \g \cap \d_1X$ of multiplicity $j_i > 1$, let us consider the $1$-forms $\{dz,\, \mathcal L_v(dz),\, \mathcal L^2_v(dz),\, \dots ,\,  \mathcal L^{j_i - 2}_v(dz)\},$ taken at the point $a_i$ (that is, view them as elements of $T^\ast_{a_i}(X)$). Then, with the help of one-parameter family of diffeomorphisms $\{\phi_{a_i, x}\}_{x \in \g}$, we spread the forms $$\{dz|_{a_i}, \mathcal L_v(dz)|_{a_i}, \mathcal L^2_v(dz)|_{a_i}, \dots ,  \mathcal L^{j_i - 2}_v(dz)|_{a_i}\}$$ along $\g$ to get $j_i -1$ independent sections $\eta_{i,0}, \eta_{i, 1}, \dots \eta_{j_i - 2}$ of $T^\ast(X)|_\g$. By their very construction, these sections are flow-invariant. Moreover, since at points of $\d_2X$ the field $v$ is tangent to $\d X = \{z = 0\}$, we get $dz(v)|_{\d_2X} = \mathcal L_v(z) = 0$. Thus $\eta_{i,0}(v)|_\g = 0$ for all $i$. 

Similarly, for each $a_i \in \d_3X$ (i.e., $j_i > 2$), the field $v$ is tangent to the manifold $\d_2X = \{z = 0,\, \mathcal L_v(z) = 0\}.$ Therefore, using the identity 
$$\mathcal L_v(dz) = v \,\rfloor\, d(dz) + d(v \,\rfloor\, dz) = d(v \,\rfloor\, dz),$$
we get $\mathcal L_v(dz)(v)|_{\d_3X} = 0$. As a result, $\eta_{i,1}(v)|_\g = 0$ for all $i$ with $j_i > 2$. Similar considerations show that for each $i$, all the sections $\{\eta_{i,k}\}_{k < j_i -1}$, have the property $\eta_{i,k}(v)|_\g = 0$---they are \emph{horizontal} $1$-forms. Therefore they can be viewed as independent sections of the subbundle $\tau^\ast(v) \subset T^\ast(X)$. With the help of $(\Gamma^\ast)^{-1}$, these sections produce independent sections of the quotient bundle $\nu^\ast(\mathcal T(v, \omega))$.

Now we take all $|\omega|'$ sections $\{\eta_{i,0}, \eta_{i, 1}, \dots \eta_{j_i - 2}\}_i$ of $T^\ast(X)|_\g$, ordered in groups  by the increasing values of $i$. For a traversally  generic $v$, by Theorem 3.3 from \cite{K2},  these sections of $\tau^\ast(v) \subset T^\ast(X)|_\g$ are linearly independent.  

As long as the combinatorial type $\omega$ of $\g$ is fixed, these sections depend smoothly on $\g$. Since their construction relies only on $\omega$, $z$, and $v$, they are globally well-defined independent sections of the conormal bundle $\nu^\ast(\mathcal T(v, \omega))$, an intrinsically defined trivialization of this  bundle. Their duals define independent sections of the normal bundle $\nu(\mathcal T(v, \omega))$.
\smallskip

The preferred orientation of each $\d_jX$, $j \geq 1$, depends only on $v|_{\d_1X}$ and the orientation of $X$. In particular, the preferred orientation of $\d_1 X$ depends on the orientation of $X$ only. As we flip the orientation of $X$, the preferred orientation of each $\d_jX$ flips as well. Therefore the preferred orientation of the tangent bundle $\tau(\mathcal T(v, \omega))$ changes, as a result of flipping the orientation of $X$, only when the cardinality of the intersection $\g \cap \d X$---the interger $|\sup(\omega)|$---is odd. 

The versal orientation of $\mathcal T(v, \om)$ behaves similarly under the change of an orientation of $X$. As a result, the polarity ``$\,\oplus$" or ``$\,\ominus$" of each component of $\mathcal T(v, \om)$ is independent of the orientation of $X$. 
\end{proof}

\begin{corollary}\label{cor2.2} For a traversally generic vector field $v$, the points of $0$-dimensional strata $\{\mathcal T(v, \om)\}_\om$ come equipped with two sets of polarities: ``$\,+, -$" and ``$\,\oplus, \ominus$".
\end{corollary}

\begin{proof} When $\omega$ has the maximal possible reduced multiplicity $|\omega|' = n$, we can compare the versal and preferred orientations at each point $\g$ of the zero-dimensional set $\mathcal T(v, \omega)$. When the two agree, we attach the polarity  ``$\oplus$" to $\g$; otherwise, its polarity  is defined to be ``$\ominus$". Of course, the preferred orientation of the normal bundle $\nu(\g, X)$ can be compared with the preferred orientation of $\d X$ at the ``lowest" point in $\g \cap \d X$. This comparison allows for another pair $(+, -)$ of polarities to be attached to $\g$.
\end{proof}
\smallskip

Our next goal is to prove that the trajectory space $\mathcal T(v)$ of a \emph{traversally generic} vector field $v$ is a {\sf Whitney stratified space} (see Definition \ref{def2.3}). Unfortunately, the proof of this claim is rather technical, so some readers may choose to proceed to Section 4. Prior to establishing, in Theorem \ref{th2.2} below, that $\mathcal T(v)$ is a Whitney stratified space, we need to prove a few lemmas.\smallskip

Recall that a function $f$ on a closed subset $Y$ of a smooth manifold $X$ is called {\sf smooth} if it is the restriction of a smooth function, defined in an open neighborhood of $Y$. 

\begin{lemma}\label{lem2.2} Let $v$ be a  traversing vector field on a compact smooth manifold $X$, and $\Gamma: X \to \mathcal T(v)$ the obvious map. Let $F \subset \mathcal T(v)$ be a closed subset and $\psi: F \to \R$ a function such that its pull-back $\Gamma^\ast(\psi)$ is smooth on $\Gamma^{-1}(F) \subset X$ (it satisfies there the property $\mathcal L_v(\Gamma^\ast(\psi)) = 0$). 

Then $\psi: F \to \R$ admits an extension $\Psi:  \mathcal T(v) \to \R$ such that  $\Gamma^\ast(\Psi)$  is a smooth function on $X$ with the property $\mathcal L_v(\Gamma^\ast(\Psi)) = 0$.
\end{lemma}

\begin{proof}  Let $h: X \to \R$ be a smooth function with the property $dh(v) > 0$.  By Corollary 4.1 from \cite{K1}, such a Lyapunov function $h$ exists for any traversing $v$. Using  $h$, we can find a finite set $\mathcal S$ of closed smooth transversal sections $\{S_\a \subset h^{-1}(c_\a)\}_\a$ of the $v$-flow, such that each trajectory hits some section from the collection $\mathcal S$. Moreover, we can assume that all the heights $\{c_\a\}$ are distinct and separated by some $\e > 0$. The set $\mathcal S$ can be given a poset structure: $\b \succ \a$ if there exists an ascending $v$-trajectory $\g$ that first pierces $S_\a$ and then $S_\b$. Evidently, this implies that $c_\a < c_\b$. 
\smallskip

For a given $\a$, consider the set $\mathcal S_{\succ\a} =_{\mathsf{def}} \{\b \succ \a\}$ and put $c^{\Uparrow}_\a =_{\mathsf{def}} \min_{\b \succ \a} \{c_\b\}$. 
\smallskip

Now the proof is an induction by the heights $\{c_\a\}$, guided by the partial order in $\mathcal S$. It is illustrated in Fig. 1.  Assume that the desired extension $$\tilde\Psi_{\succ \a}:\, h^{-1}([c^{\Uparrow}_\a,\, +\infty)) \to \R$$ of the function $\Gamma|_{\Gamma^{-1}(F)} \circ \psi$, subject to the property $\mathcal L_v(\tilde\Psi_{\succ \a}) = 0$, already has been constructed. 
The inductive step calls  for an extension of $\tilde\Psi_{\succ \a}$ to a function on  $h^{-1}([c_\a,\, +\infty))$,  while keeping it constant on the $v$-trajectories. \smallskip

Denote by $X(v, A)$ the union of $v$-trajectories through a closed subset $A \subset X$. \smallskip

Consider two sets: $F_\a =_{\mathsf{def}} \Gamma^{-1}(F) \cap S_\a$ and $Q_\a =_{\mathsf{def}} X(v,\, \coprod_{\b \succ \a} S_\b) \cap S_\a$. \smallskip

Since $\tilde\Psi_{\succ \a}$ is constant along each trajectory and $S_\a$ is smooth and transversal to the flow, $\tilde\Psi_{\succ \a}$ produces  a well-defined smooth function $\hat\Psi_{\succ \a}: Q_\a \to \R$. On the other hand, the function $\tilde\psi =_{\mathsf{def}} \psi \circ \Gamma: \Gamma^{-1}(F) \to \R$ is smooth and constant along trajectories by the lemma  hypothesis. In particular, it is a smooth function on the closed set $F_\a$. Moreover, since $\tilde\Psi_{\succ \a}$ is an extension of  $\tilde\psi$ to $h^{-1}([c^{\Uparrow}_\a,\, +\infty)) \subset X$, both functions, $\hat\Psi_{\succ \a}$ and $\tilde\psi$, agree on $F_\a \cap Q_\a$. Therefore we have produced a function $\Psi_\a^\sharp: F_\a \cup Q_\a \to \R$ which extends to a  smooth function $\tilde\Psi_\a$ on $S_\a$. In turn, $\tilde\Psi_\a: S_\a \to \R$ defines a smooth function $\hat\Psi_\a: X(v, S_\a) \to \R$ which is constant on each trajectory through $S_\a$. By their construction, $\hat\Psi_\a$ and $\tilde\Psi_{\succ \a}$ agree on the set $X(v, S_\a) \cap h^{-1}([c^{\Uparrow}_\a,\, +\infty)).$ Together, they produce a smooth function on $h^{-1}([c_\a, +\infty))$ which is constant along the trajectories through $\coprod_{\b \succeq \a} S_\b$ and extends $\tilde\psi$.
This completes the induction step.
\end{proof}

\begin{figure}[ht]
\centerline{\includegraphics[height=4in,width=4.3in]{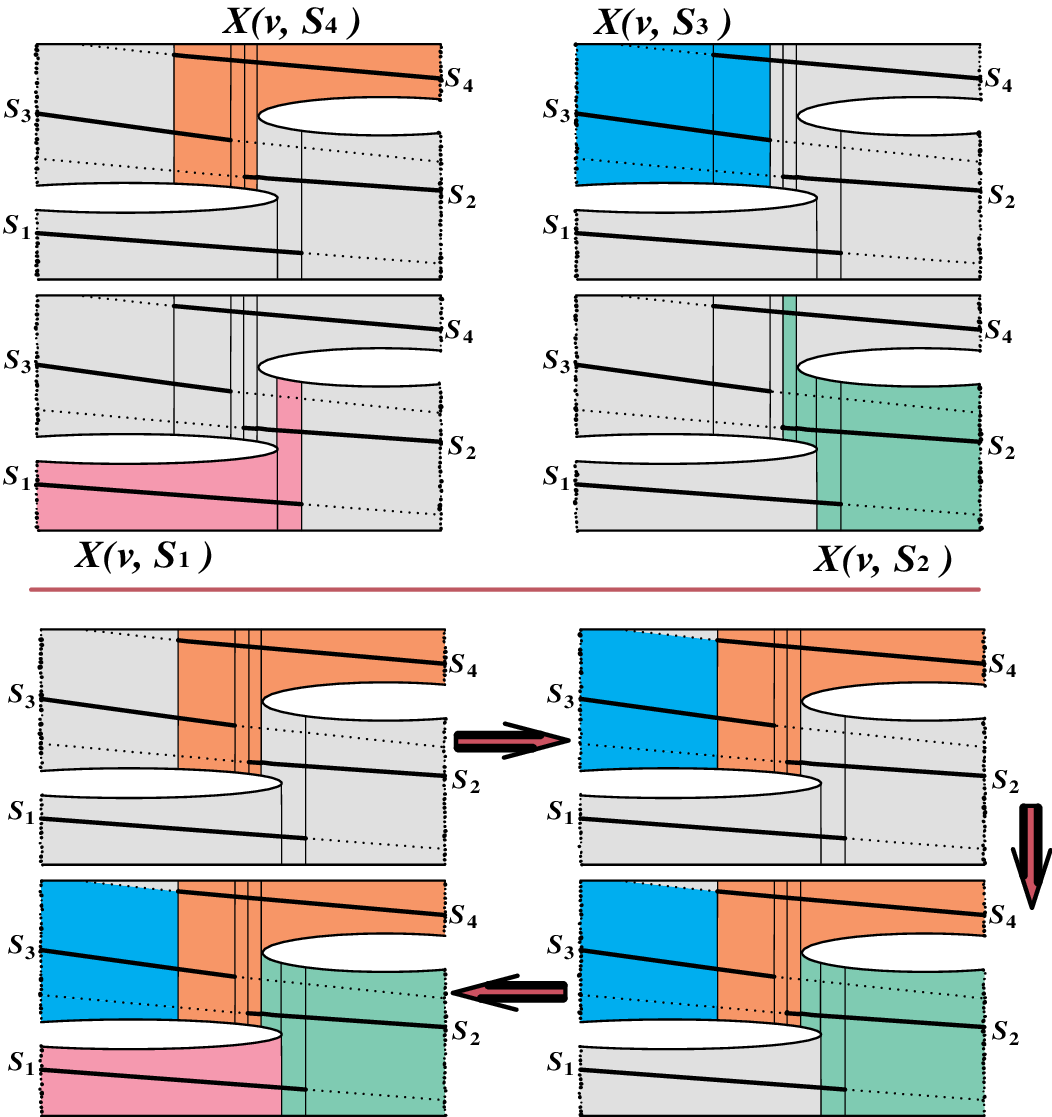}}
\bigskip
\caption{\small{The upper four diagrams show the flow sections $S_i$ and the sets $X(v, S_i)$ for $i = 1, 2, 3, 4$. The lower four diagrams show the growth of the domains of $\psi$-extensions, as they appear in the proof (to simplify the picture, the original set $\Gamma^{-1}(F)$ is not shown).}} 
\end{figure}

\begin{definition}\label{def2.3} {\bf (Whitney \cite{W})} Let $Z$ be a closed subset of a smooth manifold $M$. Consider its partition $Z = \coprod_{\a \in \mathcal S} Z_\a$, where  $\mathcal S$ a finite poset.   \smallskip

We say that $Z$ is a {\sf Whitney space} if the following properties hold:
\begin{enumerate}
\item each stratum $Z_\a$ locally is a  smooth submanifold of $M$,
\item take any pair $Z_\a \subset \bar Z_\b$ and any two of sequences $\{x_i \in Z_\b\}_i$,  $\{y_i \in Z_\a\}_i$, both converging to the same point $y \in Z_\a$.  In a local coordinate system on $M$, centered on $y$, form the secant lines $\{l_i =_{\mathsf{def}} [x_i, y_i]\}_i$ so that that $\{l_i\}_i$ converge to a limiting line $l\subset T_yM$. Also consider a sequence of tangent spaces $\{T_{x_i}(Z_\b)\}_i$ that converge to a limiting space $\tau \subset T_yM$. \smallskip

Then we require that $l \subset \tau$. \hfill $\diamondsuit$
\end{enumerate}
 \end{definition}
 
If $Z \subset M$ is a Whitney space, then one can prove that $T_y(Z_\a) \subset \tau$ (see \cite{GM2}). 
\smallskip

Now we are going to verify that the standard models of traversally  generic flows lead to spaces of trajectories which are Whitney spaces.

\begin{lemma}\label{lem2.3} Let $\omega \in \mathbf\Omega^\bullet_{' \langle n]}$. Consider the semi-algebraic set $Z_\omega := \{P_\omega(u, x) \leq 0,\, \break  \|x\| \leq \e\},$ where the polynomial $P_\omega$ of an even degree $|\omega|$ is as in (\ref{eqVERSAL}) (its real divisor has the combinatorial type $\omega$), and $\e > 0$ is sufficiently small.  Let $\mathcal T_\omega$ denote the $(\omega_{\preceq})$-stratified trajectory space of the constant vector field $v =_{\mathsf{def}} \d_u$ in $Z_\omega$. \smallskip

Then there exists an embedding $K_\omega: \mathcal T_\omega \to \R^{2|\omega|'}$, given by some smooth functions on $Z_\omega$ which are constant along each $\d_u$-trajectory that resides in $Z_\omega$.
\end{lemma}

\begin{proof} Evidently, the $x$-coordinates $x : Z_\omega \to \R^{|\omega|'}$ provide us with a map $\chi:  \mathcal T_\omega \to \R^{|\omega|'}$, given by the algebraic functions which are constant on the $\d_u$-trajectories in  $Z_\omega$. Unfortunately, $\chi$ does not separate some trajectories; that is, $\chi$ is not an embedding (just a finitely ramified map).  We will complement $x$ with another smooth map $\tilde x : Z_\omega \to \R^{|\omega|'}$, also constant on the trajectories in $Z_\omega$ and such that the pair  of maps $(x, \tilde x)$ will separate the points of $ \mathcal T_\omega$.

To construct $\tilde x$, we will use some facts from \cite{K3}, Section 4. Recall that the ball $B_\e := \{\|x\| \leq \e\}$ has a special cone structure. With the help of the Viet\'{e} map, the cone structure is  given by the local linear contractions in $\C$ of each ``near-by" divisor $D_\C(P(\sim, x_\star))$ on the ``core" divisor $D_\C(P(\sim, 0))$. This contraction produces a smooth algebraic curve \hfill\break $A_{x_\star} :  [0, 1] \to B_\e$ in the coefficient $x$-space (a generator of the ``cone"),  which  connects the given point $x_\star$ to the origin $0$. In particular, the combinatorial type of the divisor  $D_\R(P(\sim, A_{x_\star}(t)))$ is constant for all $t \in (0,1]$.

Let $S_{x_\star} =_{\mathsf{def}} \R \times A_{x_\star}$ be the ruled $(u, t)$-parametric surface that projects along the $u$-direction onto the curve $A_{x_\star}$.  Consider the intersection $\Sigma_{x_\star}$ of $S_{x_\star}$ with the set $Z_\omega$. As $x_\star \in \d B_\e$ varies, the surfaces $\{\Sigma_{x_\star}\}$ span $Z_\omega$ (the trajectory $\{x = 0\}$ serves as the binder of an  open book whose pages are the $\Sigma_{x_\star}$'s) (see Fig. 2). 

We will define a new projection $\tilde x: \Sigma_{x_\star} \to A_{x_\star}$ as follows. 
Consider the $u$-directed line $L_x$ through $x$. For a typical point $x \in A_{x_\star}$  let $\Pi_x =_{\mathsf{def}} L_x \cap \Sigma_{x_\star}$. The set $\Pi_x$ is a disjointed union of closed intervals  $\{I_i(x) = [\underline\a_i(x), \bar\a_i(x)]\}_i$ (where $\underline \a_i(x)<  \bar\a_i(x)$ are two adjacent roots of the polynomial $P(u, x)$ in (\ref{eqVERSAL})) residing in the line $L_x$. We order them so that $I_1(x) < I_2(x) < \dots < I_s(x)$ as sets (see Fig. 2, the left diagram).

Put $$\Pi_x^\vee =_{\mathsf{def}} (L_x \setminus \Pi_x) \cap [\a_{min}(x),\; \a_{max}(x)],$$  where $\a_{min}(x),\, \a_{max}(x)$ denote the minimal and the maximal real roots of the $u$-polynomial $P_\omega(u, x)$. Thus $\Pi_x^\vee$ is a finite disjoint union of closed intervals $$\{I_i^\vee(x) = [\bar\a_i(x),\, \underline\a_{i+1}(x)]\}_i,$$  residing in the line $L_x$. Note that $P_\omega \geq 0$ in each interval $I_i^\vee(x)$. We also order the intervals so that, as sets,  $$I_1^\vee(x) < I_2^\vee(x) < \dots < I^\vee_{s -1}(x).$$ 

Let $\tau_i(x)$ denote the length of the interval $I_i^\vee(x)$.\smallskip

We fix a smooth monotone function $\chi : [0, +\infty) \to [0, 1)$ such that $\chi(0) = 0$ and $\lim_{\tau \to +\infty} \chi(\tau) = 1$ (say, $\chi = \frac{2}{\pi}\tan^{-1}$). Consider a smooth $\tau$-parametric family ($\tau \in [0, +\infty)$) of smooth monotonically increasing functions $\phi_\tau: [0, 1] \to \R_+$ such that: \hfill \break (1) $0 < \phi_\tau(t) < t$ for all $t \in (0, 1]$, (2) the infinite order jet of $\phi_\tau$ of at $t=0$ coincides with the jet of the identity function $t: [0, 1] \to [0, 1]$,  (3) $\phi_\tau(1) = \chi(\tau)$, and (4) $\phi_\tau(t)$ is a smooth function in $t$ and $\tau$.

For each $i$,  we map the point $$\big(\bar\a_i\big(A_{x_\star}(t)\big),\, t\big) \in \d_1^+\Sigma_{x_\star}(\d_u)$$ to the point $$\big(\underline\a_i\big(A_{x_\star}(\phi_{\tau_i(x_\star)}(t))\big),\; \phi_{\tau_i(x_\star)}(t)\big) \in  \d_1^-\Sigma_{x_\star}(\d_u).$$  We denote by $\theta_{x_\star, i}$ this map. As a function in $(u, t)$, the map $\theta_{x_\star, i}$ is smooth. We notice that, $\phi_{\tau_i(x_\star)}(t) \neq t$ for all $t \in (0, 1]$ and $x_\star \neq \vec 0$. We also observe that, if the interval $I_i^\vee(x_\star)$ shrinks to a singleton as we vary $x_\star$, then the map $\theta_{x_\star, i}$ approaches the identity.

\begin{figure}[ht]
\centerline{\includegraphics[height=1.8in,width=2.9in]{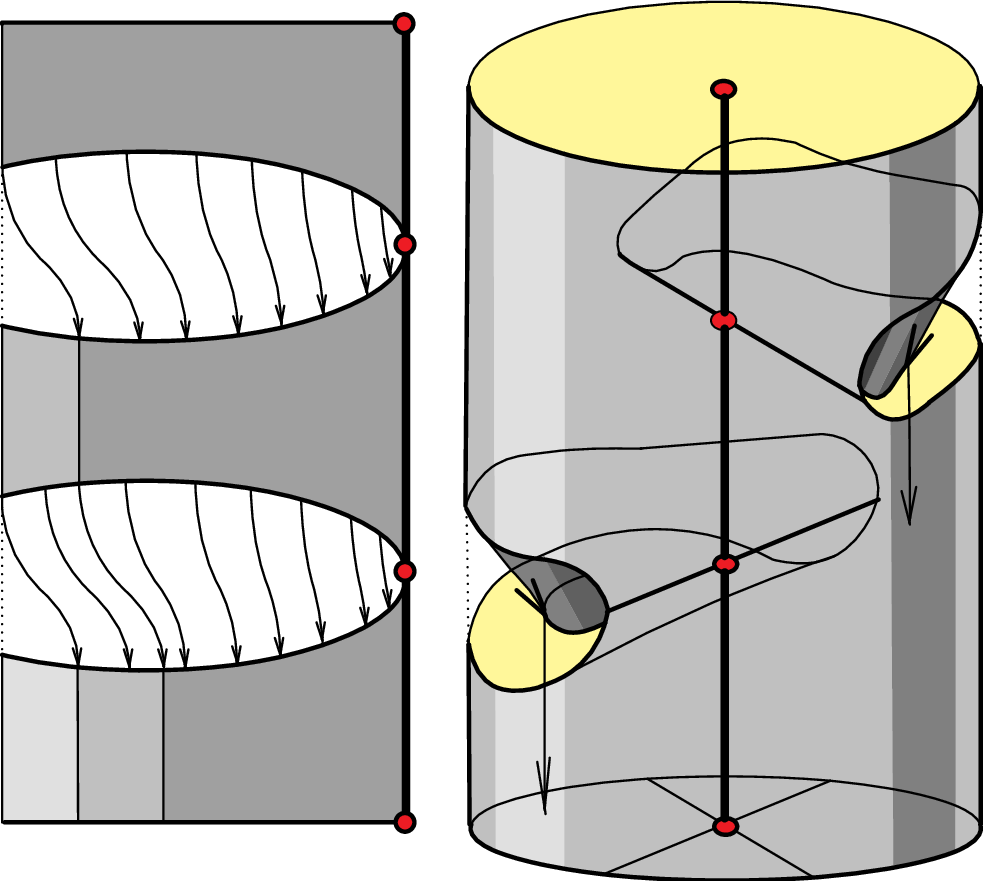}}
\bigskip
\caption{\small{The map $\tilde x:  \Sigma_{x_\star} \to A_{x_\star}$ over some arc $A_{x_\star}$ (on the left) and  the map $\tilde x : Z_\omega \to \R^{|\omega|'}$ (the deformed projection of the cylinder  with indentations on its base).}} 
\end{figure}

Now we define $\tilde x:  \Sigma_{x_\star} \to A_{x_\star}$ by the following formulas (see Fig. 2): 

\begin{eqnarray}\label{eq2.8}
\tilde x & =_{\mathsf{def}} & x \; \text{for all points in}\, I_1(x),  \nonumber\\ 
& =_{\mathsf{def}} & \theta_{x_\star, 1}\,  \, \text{for all points in}\, I_2(x),  \nonumber\\
& =_{\mathsf{def}} &  \theta_{x_\star, 2}\circ \theta_{x_\star, 1}\,  \, \text{for all points in}\, I_3(x),  \nonumber\\
& =_{\mathsf{def}} &  \theta_{x_\star, 3}\circ \theta_{x_\star, 2}\circ \theta_{x_\star, 1}\,  \, \text{for all points in}\, I_4(x),  \nonumber\\
& =_{\mathsf{def}} & \;\;\; \dots 
\end{eqnarray}

Since $0 < \phi_\tau(t) < t$ for all $t \in (0, 1]$, the map $\tilde x: Z_\omega \to \R^n$ separates the trajectories that are not distinguished by the map $x: Z_\omega \to \R^n$. Therefore the smooth map $$J_\omega=_{\mathsf{def}} (x, \tilde x): Z_\omega \to \R^{2|\omega|'},$$ being constant on each trajectory, gives rise to a a smooth (in the sense of Definition \ref{def2.2}) embedding $K_\omega: \mathcal T_\omega \to \R^{2|\omega|'}$. 
\end{proof}

\noindent{\bf Remark 2.1.} It seems that the desired embedding $K_\omega: \mathcal T_\omega \to \R^{2|\omega|'}$ cannot be delivered by analytic functions. \hfill $\diamondsuit$

\begin{figure}[ht]\label{fig8.3}
\centerline{\includegraphics[height=2.5in,width=3.8in]{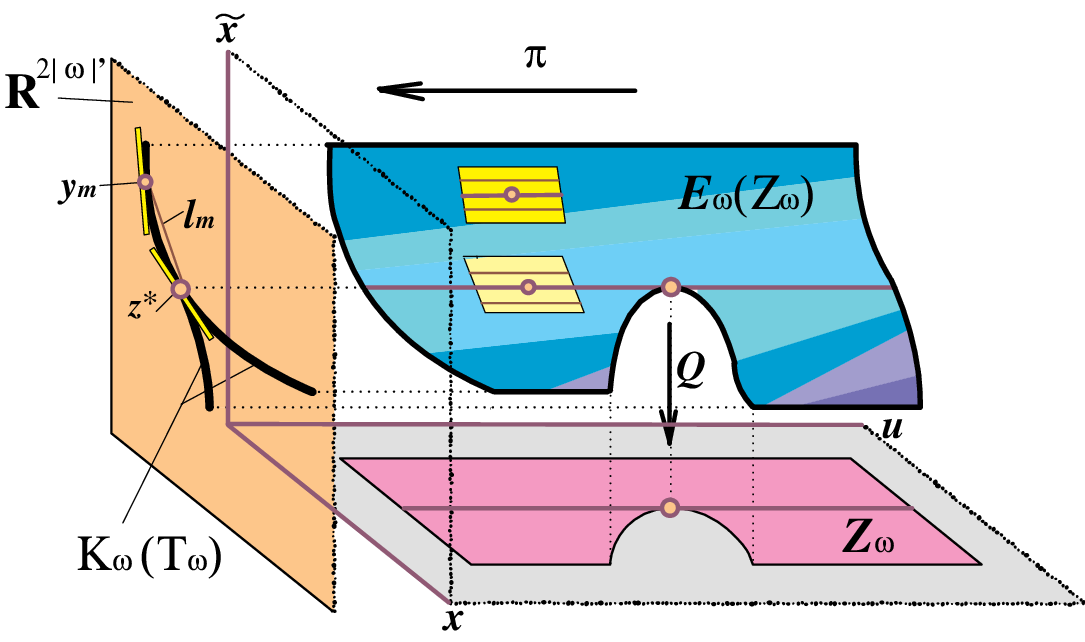}}
\bigskip
\caption{\small{The space $E_\omega(Z_\omega) \subset \R \times  \R^{2|\omega|'}$ and its projections $\pi$ and $Q$ on $K_\omega(\mathcal T_\omega)\subset \R^{2|\omega|'}$ and on $Z_\omega \subset \R \times  \R^{|\omega|'}$.}} 
\end{figure}

\begin{corollary}\label{cor2.3} The image $K_\omega(\mathcal T_\omega) \subset \R^{2|\omega|'}$ is a Whitney $(\omega_{\preceq})$-stratified space.
\end{corollary}

\begin{proof} It is useful to consult with Fig. 3 that illustrates some key elements of the proof. 

Let $\pi: \R \times \R^{2|\omega|'} \to \R^{2|\omega|'}$ denote the obvious projection. Put $K =_{\mathsf{def}} K_\omega$. Consider the map $$E =_{\mathsf{def}} E_\omega :\;  Z_\omega \to \R \times  \R^{2|\omega|'},$$ given by the formula $E(u, x) =_{\mathsf{def}} (u, J(u, x))$. Since $J =_{\mathsf{def}} J_\omega = (x, \tilde x)$, the map $E$ is a regular embedding, given by smooth functions on $Z_\omega$. Consider the projection $$Q: \R \times \R^{2|\omega|'} \to \R \times \R^{|\omega|'},$$  given by the formula $Q(u, x, \tilde x) =_{\mathsf{def}} (u, x)$. By the definition, $Q(E(Z_\omega)) = Z_\omega$.\smallskip

Let  $\mu \prec \nu$ be two elements in the poset $\omega_\preceq\subset \mathbf \Omega^\bullet$, and $\mathcal K_\mu, \mathcal K_\nu$ the two pure strata of $K(\mathcal T_\omega) \subset \R^{2|\omega|'}$, indexed by $\mu, \nu$ (thus $\mathcal K_\mu \subset \bar{\mathcal K}_\nu$). Consider a sequence of points $\{y_m \in \mathcal K_\nu\}_m$  and a  sequence of points $\{z_m \in \mathcal K_\mu\}_m$, both converging to a point $z_\star \in \mathcal K_\mu$. We need to verify that, if the tangent spaces $\{T_{y_m}\mathcal K_\nu\}_m$ converge in $\R^{2|\omega|'}$ to an affine space $T_\star$ containing $z_\star$, and the sequence of lines $\{l_m \supset [z_m, y_m]\}_m$ converges to a line $l_\star \subset \R^{2|\omega|'}$, then $l_\star \subset  T_\star$.  

Equivalently, we need to verify that if the spaces $\{\mathsf T_m =_{\mathsf{def}} \pi^{-1}(T_{y_m}\mathcal K_\nu)\}_m$ converge in $\R\times\R^{2|\omega|'}$ to an affine space $\mathsf T_\star =_{\mathsf{def}}  \pi^{-1}(T_\star) \supset \pi^{-1}(z_\star)$ (these spaces are depicted as parallelograms in Fig. 3) and the sequence of $2$-planes $\{\mathsf L_m =_{\mathsf{def}}\pi^{-1}(l_m)\}_m$ converges to a plane $\mathsf L_\star =_{\mathsf{def}} \pi^{-1}(l_\star) \subset \R\times\R^{2|\omega|'},$ then $\mathsf L_\star  \subset  \mathsf T_\star$. We call this conjectured property ``$\tilde{\mathsf B}$". 

Note that all the affine spaces $\mathsf T_m, \mathsf T_\star, \mathsf L_m$, and $\mathsf L_\star$, are fibrations with the line fibers which are parallel to the direction of $\R$ in the product $\R\times \R^{2|\omega|'}$.\smallskip

We can think of $E(Z_\omega)$ as a \emph{graph} of a smooth map $\tilde x$ from $Z_\omega$ to $\R^{|\omega|'}$. Since  $Q: E(Z_\omega)  \to Z_\omega$ is a $(\omega_\preceq)$-stratification-preserving  diffeomorphism which respects the $\d_u$-induced $1$-foliations $\mathcal F$ on $E(Z_\omega)$ and $\mathcal G$ on $Z_\omega$, the tangent spaces to the $\nu$-indexed pure stratum in $E(Z_\omega)$ are mapped by $Q$ isomorphically onto the tangent space to the $\nu$-indexed pure stratum in $Z_\omega$.  So, with the help of the graph-manifold $E(Z_\omega)$, any tangent space to the $\nu$-indexed pure stratum in $Z_\omega$ \emph{determines} the corresponding tangent space to the $\nu$-indexed pure stratum in $E(Z_\omega)$.

Let $\tilde\tau_\star$ denote the tangent space to $E(Z_\omega)$ at a generic point $\tilde z_\star \in \pi^{-1}(z_\star)$, and let  $\tau_\star$ denote the tangent space to $Z_\omega$ at the point $Q(z_\star)$. By the very definitions of $\mathsf T_\star$ and $\mathsf L_\star$ as limit objects and using that $E(Z_\omega)$ is a smooth manifold, carrying the foliation $\mathcal F$ (whose leaves are parallel lines in $\R\times\R^{2|\omega|'}$), we get that $\mathsf T_\star  \subset \tilde\tau_\star$ and $\mathsf L_\star  \subset \tilde\tau_\star$. 

Since $Q:  E(Z_\omega) \to Z_\omega$ is a diffeomorphism, $Q: \tilde\tau_\star \to \tau_\star$ is an isomorphism of vector spaces. Therefore there exist unique subspaces of  $\tilde\tau_\star$ that are mapped by $Q$ onto $Q(\mathsf T_\star)$ or onto $Q(\mathsf L_\star)$; these are exactly the spaces $\mathsf T_\star$ and $\mathsf L_\star$, respectively. Thus, $Q(\mathsf L_\star) \subset Q(\mathsf T_\star)$ if and only if $\mathsf L_\star  \subset \mathsf T_\star$.

Therefore  property $\tilde{\mathsf B}$ is equivalent to the following property $\mathsf B$:
\smallskip

``\emph{If the spaces $\{Q(\mathsf T_m )\}_m$ converge in $\R \times \R^{|\omega|'}$ to the affine space $Q(\mathsf T_\star)$, and the sequence of planes $\{Q(\mathsf L_m)\}_m$ converges to a plane $Q(\mathsf L_\star) \subset \R \times \R^{|\omega|'}$, then $Q(\mathsf L_\star) \subset Q(\mathsf T_\star)$}".
\smallskip

Note that the composition $Q \circ K:  \mathcal T_\omega \to \R^{|\omega|'}$ is delivered by employing the algebraic map $x: Z_\omega \to \R^{|\omega|'}$. The image $Q(K(\mathcal T_\omega)) \subset \R^{|\omega|'}$ is stratified by the collection of real discriminant varieties, their complements, and their multiple self-intersections, indexed by various $\mu \in \omega_\preceq$ (as described in \cite{K4}).  In particular, these strata are {\sf semi-algebraic sets}. \smallskip

By the fundamental results of \cite{Har1}, \cite{Har2}, and \cite{Hi}, the semi-analitic sets are Whitney stratified spaces. As a result, the $(\omega_{\preceq})$-stratified space $Q(K(\mathcal T_\omega))$ is a Whiney space. Thus property $\mathsf B$ is valid, since all the affine spaces, relevant to $\mathsf B$, are fibrations with the line $\pi$-fibers over the corresponding spaces in  $\R^{|\omega|'} \supset  Q(K(\mathcal T_\omega))$. Therefore, the $(\omega_\preceq)$-stratified space $K(\mathcal T_\omega)$ is a Whitney $(\omega_{\preceq})$-stratified space in $\R^{2|\omega|'}$.
\end{proof}

\begin{theorem}\label{th2.2} For a traversally  generic vector field $v$ on a $(n+1)$-dimensional $X$, the $\mathbf\Om^\bullet_{'\langle n]}$-stratified trajectory space $\mathcal T(v)$ can be given the structure of a Whitney space (residing in an Euclidean space).
\end{theorem}

\begin{proof} Let $\mathcal U =_{\mathsf{def}} \{U_r\}_r$ be a finite $v$-adjusted closed cover of $X$, such that each $U_r \subset \hat X$ admits special coordinates $(u, x, y) =_{\mathsf{def}} (u^{(r)}, x^{(r)}, y^{(r)})$ in which $\d X$ is given by the polynomial equation $\{P_r(u, x) = 0\}$ as in (\ref{eqVERSAL}).  Recall that, for a traversally generic $v$, the equation is determined by the combinatorial type $\omega_r$ of the core trajectory $\g_r \subset U_r$. 

Let us denote by $\mathcal T_r$ the space of trajectories of the  $\d_u$-flow in the domain $$U_r =_{\mathsf{def}} \{P_r(u, x) \leq 0,\, \|x\| \leq \e,\, \|y\| \leq \e'\}.$$ It is a compact subset of $\mathcal T(v)$. 

Consider the  embeddings $$K_r: \mathcal T_r \to \R^{2|\omega_r|'} \times \R^{n - |\omega_r|'}\; \text{ and } \;E_r : U_r \to \R \times  \R^{2|\omega_r|'}  \times \R^{n - |\omega_r|'},$$ given by the formulas 
$$K_r\big(\g_{\{u^{(r)},\, x^{(r)},\, y^{(r)}\}}\big) =_{\mathsf{def}} \big(x^{(r)},\; \tilde x^{(r)}(u^{(r)}, x^{(r)}),\; y^{(r)}\big),$$  
$$E_r\big(u^{(r)},\, x^{(r)},\, y^{(r)}\big) =_{\mathsf{def}} \big(u^{(r)},\, x^{(r)}, \; \tilde x^{(r)}(u^{(r)}, x^{(r)}),\; y^{(r)}\big).$$ 
Here $\g_{\{u^{(r)},\,  x^{(r)}, \, y^{(r)}\}}$ denotes the $\d_u$-trajectory in $U_r$, passing through the point $(u^{(r)}, x^{(r)},\, y^{(r)})$, and  $\tilde x^{(r)}(u^{(r)}, x^{(r)})$ is a function as in Corollary \ref{cor2.3} (see Figures 2 and 3). \smallskip

Smooth functions $\psi: \mathcal T_r \to \R$ are exactly the smooth functions on $U_r \cap X$ that are constant along the trajectories. By Lemma \ref{lem2.2}, each $\psi$ extends to a smooth function on $X$ which is constant on each trajectory.  We denote this extension $\hat\psi$.
\smallskip

Therefore, the local embeddings $\{K_r: \mathcal T_r \to \R^{2|\omega_r|'}\times \R^{n - |\omega_r|'}\}_r$ extend to some smooth maps $\{\hat K_r: \mathcal T(v) \to \R^{2|\omega_r|'}\times \R^{n - |\omega_r|'}\}_r$. Together they produce a smooth  embedding $K: \mathcal T(v) \to \R^N$, where $K =_{\mathsf{def}}  \prod_r \hat K_r$ and $\R^N =_{\mathsf{def}} \prod_r \big(\R^{2|\omega_r|'}\times \R^{n - |\omega_r|'}\big)$. 
\smallskip

Let $G: X \to \R^N$ be the composition $\Gamma \circ K$, where $\Gamma: X \to \mathcal T(v)$ is the obvious map.

We choose again a function $h:\hat  X \to \R$ such that $dh(\hat v) > 0$ in $\hat X$ (see Lemma  4.1 from \cite{K1}). With the help of $h$, we get a map $E: X \to \R \times \R^N$ given by the formula $E(z) := (h(z), G(z))$. Since  $dh(v) > 0$ and the Jacobian of each map $J_{\omega_r} =_{\mathsf{def}} (x^{(r)},\tilde x^{(r)}, y^{(r)})$ is of the maximal rank in $U_r$, the map $E$ is a regular smooth embedding. 

Composing $E$ with the obvious projection $\pi: \R \times \R^N \to \R^N$, we get the smooth (see Definition \ref{def2.2}) embedding  $K: \mathcal T(v) \to \R^N$. \smallskip

Our next goal is to show that $K(\mathcal T(v))$ is a Whitney stratified space in $\R^N$. Since Definition \ref{def2.3} of Whitney space is local, it suffices to check its validity in each local chart $\mathcal T_r \subset \mathcal T(v)$, that is, to verify that $K(\mathcal T_r) \subset \R^N$ is a Whitney space. \smallskip

The arguments below are very similar to the ones used while proving Corollary \ref{cor2.3}.

Consider the projection $p_r: \R^N \to \R^{2|\omega_r|'} \times \R^{n- |\omega_r|'}$, produced by omitting the product $\prod_{s \neq r} (\R^{2|\omega_s|'} \times \R^{n- |\omega_s|'})$ from the product $\prod_q (\R^{2|\omega_q|'} \times \R^{n- |\omega_q|'})$. Let  $$Q_r: \R\times\R^N \to \R\times\R^{2|\omega_r|'} \times \R^{n- |\omega_r|'}$$  denote the projection $id_\R \times p_r$. \smallskip

Note that the projection $Q_r$ generates a diffeomorphism between the manifold $E(U_r) \subset \R \times \R^N$ and the manifold $E_r(U_r) \subset \R \times \R^{2|\omega_r|'} \times \R^{n- |\omega_r|'}$, a diffeomorphism that respects the oriented $1$-foliations, induced by the $v$-flow on $X$, as well as the $(\omega_r)_\preceq$-stratifications of $E(U_r)$ and $E_r(U_r)$ by combinatorial types of $v$-trajectories (or rather of the $\pi$-fibers).  

We denote these foliations by $\mathcal F_r$ and $\mathcal G_r$, respectively.\smallskip

Let  $\mu \prec \nu$ be two elements in the poset $(\omega_r)_\preceq$, and let $\mathcal K_\mu, \mathcal K_\nu$ be the two pure strata of $K(\mathcal T_r) \subset \R^N$, indexed by $\mu, \nu$. Consider a sequence of points $\{y_m \in \mathcal K_\nu\}_m$  and a  sequence of points $\{z_m \in \mathcal K_\mu\}_m$, both converging to a point $z_\star \in \mathcal K_\mu$. We need to verify that, if the tangent spaces $\{T_{y_m}\mathcal K_\nu\}_m$ converge in $\R^N$ to an affine space $T_\star$ containing $z_\star$, and the sequence of lines $\{l_m \supset [z_m, y_m]\}_m$ converges to a line $l_\star \subset \R^N$, then $l_\star \subset  T_\star$.  

Equivalently, we need to verify that, if the spaces $\{\mathsf T_m =_{\mathsf{def}} \pi^{-1}(T_{y_m}\mathcal K_\nu)\}_m$ converge in $\R\times\R^N$ to an affine space $\mathsf T_\star =_{\mathsf{def}} \pi^{-1}(T_\star) \supset \pi^{-1}(z_\star)$, and the sequence of $2$-planes $\{\mathsf L_m =_{\mathsf{def}} \pi^{-1}(l_m)\}_m$ converges to a plane $\mathsf L_\star =_{\mathsf{def}} \pi^{-1}(l_\star) \subset \R\times\R^N$, then $\mathsf L_\star  \subset  \mathsf T_\star$. Let us call this conjectured property ``$\tilde{\mathsf A}$". Note that all the affine spaces $\mathsf T_m, \mathsf T_\star, \mathsf L_m$, and $\mathsf L_\star$, are fibrations with the line fibers parallel to the direction of $\R$ in $\R\times \R^N$.\smallskip

We can think of $E(U_r)$ as a graph of a smooth map from $E_r(U_r)$ to $\prod_{s \neq r} \R^{2|\omega_s|'} \times \R^{n-|\omega_s|'}$. Since  $Q_r: E(U_r)  \to E_r(U_r)$ is a stratification-preserving  diffeomorphism which respects the $v$-induced $1$-foliations $\mathcal F_r$ and $\mathcal G_r$, the tangent spaces to the $\nu$-indexed pure stratum in $E(U_r)$ are mapped \emph{isomorphically}  by $Q_r$ onto the tangent space to the $\nu$-indexed pure stratum in $E_r(U_r)$.  So, with the help of the graph-manifold $E(U_r)$, any tangent space to the $\nu$-indexed pure stratum in $E_r(U_r)$ determines the corresponding tangent space to the $\nu$-indexed pure stratum in $E(U_r)$.

Let $\tilde\tau_\star$ denote the tangent space to $E(U_r)$ at a generic point $\tilde z_\star \in \pi^{-1}(z_\star)$, and let  $\tau_\star$ denote the tangent space to $E_r(U_r)$ at the point $P_r(z_\star)$. By the very definitions of $\mathsf T_\star$ and $\mathsf L_\star$ as limit objects, and using that $E(U_r)$ is a smooth manifold carrying the foliation $\mathcal F_r$ (whose leaves are parallel lines in $\R\times\R^N$), we get that $\mathsf T_\star  \subset \tilde\tau_\star$ and $\mathsf L_\star  \subset \tilde\tau_\star$. 

Since $Q_r: E(U_r) \to E_r(U_r)$ is a diffeomorphism, $Q_r: \tilde\tau_\star \to \tau_\star$ is an isomorphism of vector spaces. Therefore there exist unique subspaces of  $\tilde\tau_\star$ that are mapped by $Q_r$ onto $Q_r(\mathsf T_\star)$ or onto $Q_r(\mathsf L_\star)$; these are exactly the spaces $\mathsf T_\star$ and $\mathsf L_\star$, respectively. Thus, $Q_r(\mathsf L_\star) \subset Q_r(\mathsf T_\star)$ if and only if $\mathsf L_\star  \subset \mathsf T_\star$.

Hence property $\tilde{\mathsf A}$ is equivalent to the following property $\mathsf A$:

\emph{If the spaces $\{Q_r(\mathsf T_m )\}_m$ converge in $\R \times \R^{2|\omega_r|'} \times \R^{n -|\omega_r|'}$ to the affine space $Q_r(\mathsf T_\star)$, and the sequence of planes $\{Q_r(\mathsf L_m)\}_m$ converges to a plane $Q_r(\mathsf L_\star) \subset \R \times \R^{2|\omega_r|'}$, then $Q_r(\mathsf L_\star) \subset Q_r(\mathsf T_\star)$}.
\smallskip

By Corollary \ref{cor2.3}, $K_r(\mathcal T_r) \subset \R^{2|\omega_r|'} \times \R^{n -|\omega_r|'}$ is a Whitney space. Therefore, property $\mathsf A$ is valid. So the property $\tilde{\mathsf A}$ has been validated as well. As a result, $K(\mathcal T(v))$ is a Whitney stratified space in $\R^N$.
\end{proof}

\noindent{\bf Remark 2.2. } It is desirable to find a more direct proof of Theorem \ref{th2.2}, the proof that will validate Whitney's property $\tilde{\mathsf B}$ geometrically, without relying on the heavy general theorems that claim: ``semi-analytic sets are Whitney spaces". In fact, the discriminant varieties in $\R^d_{\mathsf{coef}}$ that correspond to various combinatorial patterns $\omega$ for real divisors of real $d$-polynomials, do have remarkable intersection patterns for their tangent spaces and cones (see \cite{K4}). Perhaps, these properties of discriminant varieties  should be in the basis of a ``more geometrical" proof. \hfill $\diamondsuit$

\begin{corollary}\label{cor2.4} Let $X$ be an $(n + 1)$-dimensional compact smooth manifold, carrying a traversally generic vector field $v$. Then the following claims are valid:
\begin{itemize}
\item The space of trajectories $\mathcal T(v)$ admits the structure of finite cell/simplicial complex. 
\item For each  $\omega \in \mathbf \Omega^\bullet _{' \langle n]}$, the stratum $\mathcal T(v, \omega_{\succeq_\bullet})$ is a  codimension $|\omega|'$ subcomplex of $\mathcal T(v)$. 
\item With respect to an appropriate cellular/simplicial structure in $X$, the obvious map $\Gamma: X \to \mathcal T(v)$ is cellular/simplicial. 
\item Moreover, $\Gamma: X \to \mathcal T(v)$ is a homotopy equivalence.
\end{itemize}
\end{corollary}

\begin{proof} By Theorem \ref{th2.2}, the trajectory space $\mathcal T(v)$ of a traversally generic flow admits a structure of a Whitney space  embedded in some ambient Euclidean space. 

The fundamental results of \cite{Gor}, \cite{Jo}, and \cite{Ve} claim that the Whiney spaces $Y$ admit smooth triangulations $\tau: T \to Y$, amenable to their stratifications. The adjective ``smooth" here refers to the homeomorphism $\tau$ being smooth on the interior of each simplex $\D$ (remember, for a traversally generic $v$, the pure strata $\mathcal T(v, \omega)$ are smooth manifolds!).  With respect to such triangulations, the strata are subcomplexes. 
Therefore $\mathcal T(v)$ admits a finite triangulation so that each stratum $\mathcal T(v, \omega_{\succeq_\bullet})$ is a subcomplex.

For traversing vector fields $v$, over each open simplex $\D^\circ \subset \mathcal T(v)$, the map $\Gamma: X \to \mathcal T(v)$ is a trivial fibration whose fibers are either closed segments, or singletons. Thus each set $\Gamma^{-1}(\D^\circ)$ is homeomorphic either to the cylinder $\D^\circ \times [0, 1]$, or to $\D^\circ$. This introduces a cellular structure on $X$ such that $\Gamma$ becomes a cellular map. With a bit more work, one can refine the cellular structures in $X$ and $\mathcal T(v)$, so that $\Gamma$ becomes a simplicial map. 

Since, by Theorem 5.1 from \cite{K1}, $\Gamma: X \to \mathcal T(v)$ is a weak homotopy equivalence and both spaces are $CW$-complexes, we conclude that $\Gamma$ is a homotopy equivalence \cite{Wh1}.
\end{proof}

\noindent{\bf Remark 2.3.} Most probably, $\mathcal T(v)$ is a compact $CW$-complex for any traversing and boundary generic (and not necessary \emph{traversally} generic) vector field $v$. However, for such vector fields, we do not have the ``open book" algebraic models (as in Fig. 2) for their interactions with boundary $\d_1X$ in the vicinity of a typical trajectory. So we do not know how to extend the previous arguments to a larger class of traversing vector fields. 
\hfill $\diamondsuit$
\smallskip

Next, we introduce one key construction from \cite{K8} which will turn out to be very useful throughout our investigations, especially in proving the Holography Theorem \ref{th3.HOLO}.
\begin{figure}[ht]
\centerline{\includegraphics[height=1.8in,width=2.9in]{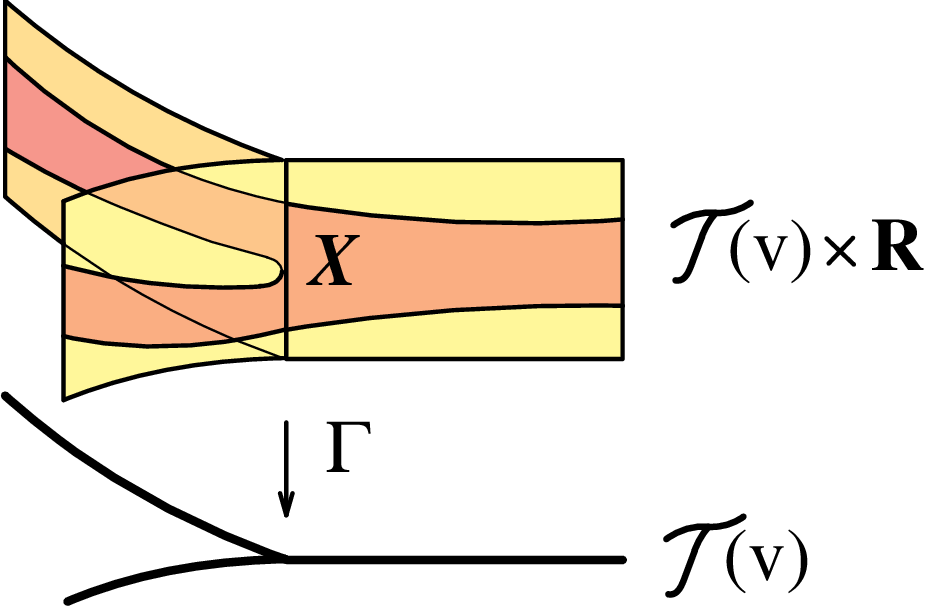}}
\bigskip
\caption{\small{The embedding $\a(f, v)$ of $X$ into the product $\mathcal T(v) \times \R$.}} 
\end{figure}
\begin{lemma}\label{lem3.1} For any traversing vector field $v$ on $X$, there is an embedding $\a: X \subset \mathcal T(v) \times \R$. In fact, any pair $(f, v)$ such that $df(v) > 0$  generates such an embedding $\a = \a(f, v)$ in a canonical fashion. 

For any smooth map $\b: \mathcal T(v) \to \R^N$, the composite map $$A(v, f): X \stackrel{\mathcal \a}{\longrightarrow} \mathcal T(v) \times \R \; \stackrel{\b \times id}{\longrightarrow}\; R^N \times \R$$ is smooth.
\smallskip

Any two embeddings $\a(f_1, v)$ and $\a(f_2, v)$ are isotopic through homeomorphisms, provided that $df_1(v) > 0, df_2(v) > 0$.
\end{lemma}

\begin{proof} We know that a traversing $v$ admits a Lyapunov function $f$. Since $f$ is strictly increasing along the $v$-trajectories, any point $x \in X$ is determined by the $v$-trajectory $\g_x$ through $x$ and the value $f(x)$.  Therefore, $x$ is determined by  the point $\g_x \times f(x) \in \mathcal T(v) \times \R$. By the definition of topology in $\mathcal T(v)$, the correspondence $\a(f, v): x \rightarrow \g_x \times f(x)$ is a continuous map. 

In fact, $\a(f, v)$ is a smooth map in the spirit of Definition \ref{def2.1}: more accurately,  for any map $\b: \mathcal T(v) \to \R^N$, given by $N$ smooth functions on $\mathcal T(v)$, the composite map $A(v, f): X \to \R^N \times \R$ is smooth.  The verification of this fact is on the level of definitions. 
\smallskip

For a fixed $v$, the condition $df(v) > 0$ defines an open \emph{convex} cone $\mathcal C(v)$ in the space $C^\infty(X)$. Thus, $f_1$ and $f_2$ can be linked by a path within the space of nonsingular functions on $X$, which results in  $\a(f_1, v)$ and $\a(f_2, v)$ being homotopic through homeomorphisms. 
\end{proof}

\noindent{\bf Remark 2.4.} By examining  Figure 4, we observe an interesting phenomenon: the embedding $\a: X \subset \mathcal T(v) \times \R$ does not extend to an embedding of a larger manifold $\hat X \supset X$, where $\hat X \setminus X \approx \d_1X \times [0, \e)$. In other words, $\a(\d_1X)$ has no outward ``normal field" in the ambient $\mathcal T(v) \times \R$; in that sense, $\a(\d_1X)$ is  \emph{rigid} in $\mathcal T(v) \times \R$! \hfill $\diamondsuit$




\section{The Causality-based Holography Theorems}

Now we are in position to formulate the question in the center of this paper:

{\it ``Is it is possible to reconstruct a manifold $X$ and a traversing $v$-flow on $X$ from some $v$-generated data, available on the boundary $\d X$?"}
\smallskip

When such a reconstruction is possible (see Theorem \ref{th3.HOLO} and Corollary \ref{cor3.9}), the corresponding proposition deserves the adjective ``{\sf holographic}" in its name\footnote{We own an apology to the fellow physisits: the name does not suggest a direct connection to the holography principles in the quantum field theory and the dual theories of gravity.}.

Given a traversing field $v$ on $X$, consider the map $$C_v: \d_1^+X(v) \to \d_1^-X(v)$$ that takes any point $x \in \d_1^+X$ to the next point $y$ from the set $\g_x \cap \d_1 X$, the order on the trajectory $\g_x$ being defined by $v$. We call $C_v$ the {\sf causality map} of $v$ (see Theorem \ref{th3.6} for a justification of the name). 
\smallskip

Of course the traversing fields have no closed trajectories. Nevertheless, in the world of such fields on manifolds $X$ with boundary, the causality map can be thought as a weak \emph{substitute} for the Poincar\'{e} return map (see \cite{Te} for the definition of the Poincar\'{e} return map). The dynamics of $C_v$ under (finitely many) iterations reflects the \emph{concavity} of $X$ with respect to the $v$-flow (see \cite{K1}). The ``iterations" of $C_v$ are only \emph{partially-defined} maps. \smallskip

We are already familiar with the \emph{discontinuous} nature of $C_v$. Implicitly, it animates the investigations in  \cite{K1}-\cite{K4}.  The bright spot is that $C_v$ is {\sf semi-continuous} relative to any nonsingular function $f: X \to \R$ with the property $df(v) > 0$. This semi-continuity has the following manifestation: for any $x \in \d_1^+X$ and $\e > 0$, there is a neighborhood $U_\e(x) \subset \d_1X$ such that 
\begin{eqnarray}\label{eq4.1}
 f(C_v(y)) - f(y) & \geq  &0,  \;\; \text{and} \\ \nonumber
 f(C_v(y)) - f(y) & > & f(C_v(x)) - f(x) - \e\;\;  \text{for all}\;  y \in U_\e(x).
\end{eqnarray}

Note that $C_v(x) = x$ exactly when $x \in \d_2^-X^\circ \cup \d_3^-X^\circ \cup  \dots \cup \d_{n+1}^-X $. 
\smallskip

We may take alternative and more formal view of the map $C_v$. 

Note that a traversing $v$-flow on $X$ defines the structure of a {\sf partially ordered set} on $\d X$: we write $x \prec x'$, where $x, x' \in \d X$, if there is an ascending $v$-trajectory (not a singleton) that connects $x$ to $x'$.  Let us denote by $\mathcal C^\d(v)$ this poset $(\d X, \prec)$. Evidently,   $x \preceq x'$ if and only if $x'$ is an image of $x$ under a number of iterations of the causality  map $C_v$, provided $v$ being boundary generic. Therefore, the poset $\mathcal C^\d(v)$ allows for a reconstruction of the causality  map $C_v$.
\smallskip



\smallskip
\noindent{\bf Remark 3.1.}
 Note that Lemma 3.4  and formula $(3.19)$ from \cite{K2} provide, among other things,  for \emph{local models of the causality  maps}  $C_v$, generated by traversally  generic fields $v$. In the special coordinates $(u, x, y)$, $C_v$ amounts to taking each root of the $u$-polynomial $P(u, x)$, residing in a maximal interval $I(x)$ where $P(u, x) \leq 0$, either to the next root residing in $I(x)$, or to itself (when $I(x)$ happens to be a singleton). By Theorem 2.2 from \cite{K2}, this is a map from the semi-algebraic set $\{P(u, x) = 0, \, \frac{\d P}{\d u}(u, x) \geq 0, \, \|x\| \leq \e\}$ to the  the semi-algebraic set $\{P(u, x) = 0, \, \frac{\d P}{\d u}(u, x) \leq 0, \, \|x\| \leq \e\}$.  These observations  form a foundation of the notion of {\sf holographic structure} on $\d X$, a subject of future investigations. 
 
 \hfill $\diamondsuit$
\smallskip


For a traversing field $v$, the smooth functions on $X$ that are constant along each $v$-trajectory $\g$ give rise to smooth functions on $\d_1X$.  Such functions are constant along each $C_v$-trajectory $\g^\d = \g \cap \d_1X$. Furthermore, any smooth function on $\d X$ which is constant on each finite set $\g^\d$ gives rise to a unique \emph{continuous} function on $X$, which is constant along each trajectory $\g$. However, such functions may not be automatically smooth on $X$!  

For a traversing $v$, consider the algebra $\mathsf{Ker}(\mathcal L_v) \approx  C^\infty(\mathcal T(v))$ of smooth functions on $X$ that are constants along each $v$-trajectory. 

\begin{question}\label{q4.1} Given a boundary generic traversing vector field $v$ on $X$, how to characterize the image (trace) of the algebra $\mathsf{Ker}(\mathcal L_v)$ in the algebra $C^\infty(\d X)$ in terms of the causality map $C_v$? \hfill $\diamondsuit$
\end{question}

 Let $\mathcal L_v^{(k)}$ be the $k$-th iteration of the Lie derivative in the direction of the field $v$. For a boundary generic field $v$, we denote by $\mathsf m_j(v)$  the algebra of smooth functions $\psi$ on $\d X$ such that $(\mathcal L_v^{(k)}\psi) \big |_{\d_{k+1}X} = 0$ for  all $k \in [1, j]$.  Let us denote by $\mathsf m_j(v)^{C_v}$ the subalgebra of functions from $\mathsf m_j(v)$ that are constants on each $C_v$-trajectory $\g^\d = \g \cap \d X$. 
It is easy to check that $\mathsf{Ker}(\mathcal L_v)|_{\d X} \subset \mathsf m_n(v)^{C_v}$, however, the validity of the converse claim is not obvious.
\smallskip

Temporarily we move away from the category smooth maps towards the category of piecewise differentiable (``$\mathsf{PD}$" for short) maps.

\begin{definition}\label{def4.1} We say that a triangulation $\mathsf T^\d$ of $\d X$ is invariant under the causality  map $C_v: \d^+_1X \to \d^-_1X$, if the interior of each simplex  from $\mathsf T^\d$ is mapped homeomorphically by $C_v$ onto the interior of a simplex\footnote{Remember, $C_v$ is typically a discontinuous map!}.  \hfill $\diamondsuit$
\end{definition}

\begin{lemma}\label{lem4.2} If $v$ is a traversally  generic field on $X$, then the boundary $\d_1 X$ admits a $C_v$-invariant smooth triangulation. 
\end{lemma}

\begin{proof} For boundary generic vector fields $v$, the map $\Gamma: \d_1 X \to \mathcal T(v)$ is finitely ramified surjection. For a traversally generic $v$, the lemma follows from Corollary \ref{cor2.4}: any triangulation of the trajectory space $\mathcal T(v)$, consistent with its $\mathbf\Omega^\bullet$-stratification, with the help of $\Gamma^{-1}$, lifts to a triangulation $\mathsf T^\d$ of $\d_1 X$. Indeed, for each $\omega$, by Corollary 5.1 from \cite{K3}, the map $$\Gamma: \Gamma^{-1}(\mathcal T(v, \omega)) \cap \d_1 X \to \mathcal T(v, \omega)$$ is a trivial covering. By its very construction, the triangulation $\mathsf T^\d$ is $C_v$-invariant. 
\end{proof}

\noindent{\bf Remark 3.2.} The existence of a triangulation on $\d_1 X$ by itself does not imply the existence of a triangulation on $\mathcal T(v)$: there are smooth manifolds that can serve as finite covering spaces over topological manifold bases that do not admit any triangulation! For example, the standard sphere may cover a non-triangulable fake real projective space (see \cite{CS}). \hfill $\diamondsuit$
\smallskip

\begin{figure}[ht]
\centerline{\includegraphics[height=1.5in,width=4in]{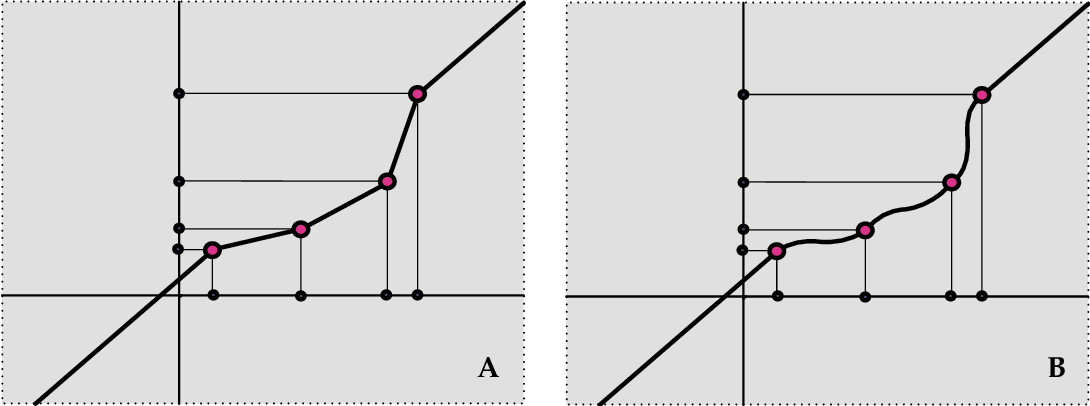}}
\bigskip
\caption{\small{The $\mathsf{PL}$ and smooth canonical interpolating homeomorphisms $\phi_{\vec x, \vec y}: \R \to \R$ that map a given sequence of 4 distinct numbers $\vec x$ to a given sequence $\vec y$ of 4 distinct numbers.}} 
\end{figure}

\noindent{\bf Remark 3.3.} Recall that, by  Whitehead's Theorem \cite{Wh}, 
any smooth manifold admits a unique $\mathsf{PD}$-structure (consistent with its differentiable structure). Therefore, different $C_v$-invariant smooth triangulations $\{\mathsf T^\d\}$ of $\d_1 X$ all are  $\mathsf{PD}$-equivalent, but perhaps not as $C_v$-\emph{invariant} triangulations!  In other words, a common refinement of two $C_v$-invariant differentiable triangulations of $\d_1X$ may be not $C_v$-invariant. \smallskip

We conjecture that any two smooth $C_v$-invariant triangulations have a $C_v$-invariant smooth refinement. That is, the trajectory space $\mathcal T(v)$ admits a \emph{unique} $\mathsf{PD}$-structure that is consistent with the preferred $\mathsf{PD}$-structure on the smooth manifold $\d_1 X$. 
\hfill $\diamondsuit$
\smallskip

Recall again that a function $f$ on a closed subset $Y$ of a smooth manifold $X$ is called {\sf smooth} if it is the restriction of a smooth function, defined in an open neighborhood of $Y$. \smallskip

Let $v$ be a traversing field on a compact manifold $X$, and $A \supset B$ two closed subsets of $\d_1X$. We denote by $X(v, A)$ and $X(v, B)$ the sets of $v$-trajectories through $A$ and $B$, respectively.\smallskip

To prove Theorem \ref{th3.HOLO} below, we need the following lemma.

\begin{lemma}\label{lem4.4} Let $v$ be a traversing and boundary generic vector field on a compact manifold $X$ and $A \subset B \subset \d X$ closed subsets. Consider a smooth function $$f: B \cup X(v, A) \to \R$$ such that $f(x) < f(x')$ for any two points $x \neq x'$ on the same trajectory, such that $x'$ can be reached from $x$ by moving along the trajectory in the direction of $v$. \smallskip

Then $f$ extends to a smooth function $F: X(v, B) \to \R$ such that $\mathcal L_v(F) > 0$ on $X(v, B)$.
\end{lemma}

\begin{proof}The argument is an induction by the increasing combinatorial types $\om \in \mathbf\Om^{\bullet}$ of the $v$-trajectories that pass trough the points of the set $B \setminus A$. With $A$ being fixed, we intend to increase gradually the locus $\tilde B \supset A$ to which the desired extension exists, until eventually $\tilde B$ will coincide with the given $B$. \smallskip

In the proof, we put $X(A) := X(v, A)$ and $X(B) := X(v, B)$. Since $A$ and $B$ are closed in $\d_1X$, both sets $X(A)$ and $X(B)$ are compact. 
\smallskip
 
Thanks to the property of $v$ to be boundary generic, the set of combinatorial types $\om$ of the $v$-trajectories in $X$ is finite. So we may assume that, for some even $d$, all the elements $\om$ have the property $|\om| \leq d$. \smallskip

Consider all the trajectories through the points of $B \setminus A$ and their combinatorial types, which reside in the finite set $\Theta_B$.  
Among these types, we pick a \emph{minimal} element $\om$.  \smallskip

Denote by $X_\om$ the subset of $X(B)$ that is formed by the trajectories of this minimal  combinatorial type $\om$. Let $X^\d_\om =_{\mathsf{def}} X_\om \cap \d_1 X$. 
We denote by $\mathcal T_\om$ the $\Gamma$-image of $X_\om$ and by $\mathcal T(A)$ the $\Gamma$-image of $X(A)$.\smallskip

By the choice of minimal $\om$, the trajectories that are the limits of trajectories from $X_\om$, but are not contained in $X_\om$, have combinatorial types residing in the sub-poset $\om_{\succ_\bullet} \cap \Theta_B$ and thus are contained in $X(A)$.  
\smallskip


We are going to show that any given smooth function $$f : X^\d_\omega \cup X(A) \to \R,$$ with the properties as in the lemma, extends to smooth function $$F:  X_\omega \cup X(A)  \to \R$$ so that $\mathcal L_v(F) > 0$ on $X_\omega \cup X(A)$. Since replacing $f$ with $f + const$ produces an equivalent extension problem,  we may assume without lost of generality that $f > 0$ . \smallskip

First we notice that, for each trajectory $\g \subset  X_\om$, there is a smooth strictly monotone function $F_\g: \g \to \R$ that  takes the given increasing values of the discrete function $f|_\g: \g \cap X^\d_\om \to \R$. 
This interpolating construction is based on a standard monotone block-function $\varphi_{a,b}: [0, a] \to [0, b]$ that smoothly depends on the two non-negative parameters $a, b$. The infinite jet of $\varphi_{a,b}$ at $0$ coincides with the jet of the function $x$, the infinite jet of $\varphi_{a,b}$ at $a$ coincides with the jet of the function $x + b$, and $\frac{d}{dx}\varphi_{a,b}(x) > 0$ in the interval $[0, a]$. Fig. 5, 
diagram B, shows the four-points interpolation $\phi_{\vec x, \vec y}$ that uses three block-functions of the type $\varphi_{a,b}$.\smallskip

Since we have chosen $f > 0$, we get $F_\g > 0$ as well. 
\smallskip

Let $V(A)$ be an open regular neighborhood of $\mathcal T(A)$ in $\mathcal T(v)$. Put $U(A) := \Gamma^{-1}(V(A))$. 

Let $\tilde f_A: U(A) \to \R$ be a smooth extension of $f: X(A) \cup B \to \R$ into a neighborhood $U(A)$ of $X(A)$. We choose $V(A)$ so small that, by continuity,  $d\tilde f_A(v) > \delta > 0$ and $\tilde f_A > 0$ in $U(A)$. Also by the very construction of the extension $\tilde f_A$, its restriction to $U(A) \cap B$ coincides with $f$. \smallskip

The two sets $\mathcal T_\om$ and $V(A)$ form an open cover of the space $\mathcal T_\om \cup V(A)$. \smallskip

Let $W =_{\mathsf{def}} \mathcal T_\om \cap (V(A) \setminus \mathcal T(A))$ and $K =_{\mathsf{def}} \Gamma^{-1}(W)$. Then $\Gamma: K \to W$ is a $v$-oriented fibration with fibers being closed segments or singletons. So it is a trivial fibration. At the same time, $\Gamma: K \cap \d_1X \to W$ is a finite cover with the fiber of cardinality $\sup (\om) = |\om| - |\om|'$. The triviality of $\Gamma: K \to W$ implies that the holonomy of the covering map $\Gamma: K \cap \d_1X \to W$ is trivial and thus $K \cap \d_1X$ is a product. So $K \cap \d_1X$ is homeomorphic to $L \times C$, where $C$ is a set of cardinality $\sup (\om)$, $L \subset \d_1X \cap X_\om$, and $\Gamma: L \to W$ is a homeomorphism. 

By the construction of $K$, its boundary consists of two disjoint compacts: $\d'K$ that resides in $X(A)$ and $\d''K = X(\om) \cap \d U(A)$.  
\smallskip

We claim that, for any $\e > 0$, there exists a smooth function $\psi^\bullet_\e$ in the vicinity of $K$ in $\hat X$ such that: 1) $jet^\infty(\psi^\bullet_\e)|_{\d'K} = jet^\infty(\mathbf 1)|_{\d'K}$, 2) $jet^\infty(\psi^\bullet_\e)|_{\d''K} = jet^\infty(\mathbf 0)|_{\d''K}$, and 3) $|\mathcal L_v(\psi^\bullet_\e)| < \e$ in $K$.\footnote{We suspect that, in fact, there exists a smooth $\psi^\bullet$ that satisfies 1) and 2) and $\mathcal L_v(\psi^\bullet) = 0$ in $K$. Its existence would simplify the following arguments.} Indeed, by Whiney's version of the Urysohn lemma, there exists a smooth $\tilde\psi^\bullet: K \to [0, 1]$ such that 1) and 2) are satisfied. Since the $v$-flow gives $K$ a product structure $I \times L \subset \R \times L$, there exists a stretching diffeomorphism $\a: K \to \tilde K \approx \tilde I \times L$ (where $\tilde I \supset I$) along the $v$-directed component $I$ so that the $v$-directional derivative of $(\a^{-1})^\ast(\tilde\psi^\bullet)$, being restricted to $K \subset \tilde K$, is $\e$-small. So $\psi^\bullet_\e := (\a^{-1})^\ast(\tilde\psi^\bullet)|_K$ satisfies property 3) as well.
\smallskip

Let $\psi^A_\e: U(A) \to [0, 1]$ be the function that equals $\mathbf 1$ on $X(A)$, equals $\psi^\bullet_\e$ on $K$, and is $\mathbf 0$ on $X_\om \setminus (X_\om \cap U(A))$. It is smooth thanks to the properties $jet^\infty(\psi^\bullet_\e)|_{\d'K} = jet^\infty(\mathbf 1)|_{\d'K}$ and $jet^\infty(\psi^\bullet_\e)|_{\d''K} = jet^\infty(\mathbf 0)|_{\d''K}$. Let $\psi^\om_\e := \mathbf 1 - \psi^A_\e$. The pair $\{\psi^A_\e, \psi^\om_\e\}$ is a smooth partition of unity, subordinate to the open cover $$\{X_\om \setminus (\bar X_\om \cap X(A)),\; X(A) \cup K\}$$ of the compact space $X_\om \cup X(A)$. 
\smallskip 

Next, we form the smooth function $F_\om: X_\om \setminus (X_\om \cap X(A)) \to \R$ whose restriction $F_\g: \g \to \R$ to each $v$-trajectory $\g \subset X_\om$ is a monotone function, the canonical interpolation (see Fig. 5, 
B) of the given function $f|_\g: \g \cap X^\d_\om\to \R$. 

Consider the smooth function $\tilde F_\e: K \to \R$, defined by the formula $$\tilde F_\e := \psi_\e^\om \cdot F_\om + \psi_\e^A \cdot \tilde f_A.$$ It smoothly extends in the obvious way to a function $F_\e: X(A) \cup X_\om \to \R$ so that 1) $F_\e = F_\om$ on $X_\om \setminus (X_\om \cap U(A))$, 2) $F_\e = f$ on $X(A) \cup (X_\om \cap \d_1X)$.

By choosing an appropriate $\e$, we aim to insure that $d F_\e(v) > 0$ in $X(A) \cup X_\om$. Evidently $d F_\e(v) > 0$ in the complement to $K$. So we need to concentrate on $d F_\e(v)|_K$.  Let $\bar K$ be the closure of $K$ in $X$.

Put $m =_{\mathsf{def}} \min\{\min_{\bar K} dF_\om(v), \min_{\bar K} d\tilde f_A(v)\}$. By the properties of $F_\om$ and $\tilde f_A$, we have $m > 0$. Then by the product rule,
$$\mathcal L_v(\tilde F_\e) = d F_\e(v)\; \geq \; m + d\psi_\e^\om(v) \cdot F_\om + d\psi_\e^\A(v) \cdot \tilde f_A.$$ 
So it suffices to insure that RHS of this inequality is positive in order to guarantee that $\mathcal L_v(F_\e) > 0$ in $\bar K$. Since $d\psi_\e^\A(v) =-d\psi_\e^\om(v)$ on $\bar K$, the last inequality may be written as 
$$m + d\psi_\e^\om(v) (F_\om - \tilde f_A) > 0.$$  
Using that $\tilde f_A > 0$ and $F_\om > 0$, the choice $\e <  \inf_{\bar K}\big\{\frac{m}{|F_\om - \tilde f_A |}\big\}$ validates that $d F_\e(v) > 0$ in $X(A) \cup X_\om$. For such a choice of $\e >0$, the smooth function $F =_{\mathsf{def}} F_\e$ delivers the desired extension. 
\smallskip

Finally, we form the closed set  $A' =_{\mathsf{def}} A \cup X^\d_ \omega \subset \d_1X$ and apply the previous arguments to the new pair $B \supset A'$. This completes the inductive step $A \Rightarrow A \cup (X_\om \cap \d_1X)$. \hfill 
\end{proof}

By letting $A = \emptyset$ and $B = \d_1X$ in Lemma \ref{lem4.4}, we get an instant implication:

\begin{corollary}\label{cor4.2} Let $v$ be a traversing and boundary generic vector field on a compact smooth manifold $X$. Consider a smooth function $f: \d X \to \R$ such that $f(x) < f(x')$ for any two points $x \neq x'$ on the same trajectory, such that $x'$ can be reached from $x$ by moving in the direction of $v$. 

Then $f$ extends to a smooth function $F: X \to \R$ such that $\mathcal L_v(F) > 0$ on $X$. \hfill $\diamondsuit$
\end{corollary}

If the following conjecture (linked to Question \ref{q4.1}) is true, it would strengthen Theorem \ref{th3.HOLO} below.

\begin{conjecture}\label{conj3.?} 
Let $\d_1X \subset \R \times \R^n$ be a smooth hypersurface, given by a polynomial equation 
$$P(u, \vec x) =_{\mathsf{def}} u^d + \sum_{i=0}^{d-1} x_i u^i = 0$$
of an even degree $d$, and $X$ be the domain, given by  the polynomial inequality $\{P(u, \vec x) \leq 0\}$. 

We denote by $\g(\vec x) \subset \R \times \{\vec x \}$ a segment/singleton with the following two properties:
\begin{itemize}
\item $P(u, \vec x)|_{\g(\vec x)} \leq 0$, and 
\item no larger segment  $\tilde\g(\vec x) \supset \g(\vec x)$ has the property $P(u, \vec x)|_{\tilde\g(\vec x)} \leq 0$
\end{itemize}  

Consider a smooth diffeomorphism $\phi: \d_1X \to \d_1X$ which maps each set $\g(\vec x) \cap \d_1X$ to a similar set $\g'(\vec x') \cap \d_1X$, while preserving the multiplicity of the $P$-roots and their order in the two sets.  
 
 Let $F: X \to \R$ be a smooth function such that $\frac{\d F}{\d u} = 0$ in $X$. We denote by $f$ the restriction of $F$ to $\d_1X$.\smallskip
 
 Then the function $\phi^\ast(f)$ extends to a \emph{smooth} function $G: X \to \R$ such that $\frac{\d G}{\d u} = 0$ in $X$.
\hfill $\diamondsuit$
\end{conjecture}

Here is a special case of Conjecture \ref{conj3.?} that we can validate. 
It is the case of a boundary generic vector field in the vicinity of $\d_2^-X(v)$. 

Let $Q$ denote the hypersurface $\{u^2 + x_0 = 0\}$ in $\R^1 \times \R^1 \times \R^{n-1}$.  The functions $u: Q \to \R$ and $\vec y: Q \to \R^{n-1}$ are smooth coordinates on $Q$. 
Let $X \subset \R^1 \times \R^1 \times \R^{n-1}$ be the domain defined by $\{u^2 + x_0 \geq 0\}$.

The causality map $\a =_{\mathsf{def}} C_{\d_u}$ takes each point $q = (u, x_0, \vec y) \in Q$ to the point $\a(q) = (-u, x_0, \vec y)$. 

We denote by $K \subset Q$  the locus $\{u = 0\}$  and by $\pi: \R \times \R^n \to \R^n$ the projection $(u, x_0, \vec y) \to (x_0, \vec y)$. 

\begin{lemma}\label{lem3.A} Let a function $f: Q \to \R$ be of the class $C^{2k}(Q, \R)$ and invariant under the involution $\a: Q \to Q$. Then there exists a function $g: (\R^n)_+ \to \R$ in the variables $(x_0, \vec y)$ such that:
\begin{itemize} 
\item  the restriction of $\pi^\ast(g)$ to $Q$ coincides with $f$,

\item $g \in C^{k}((\R^n)_+, \R)$.
\end{itemize}
\end{lemma}

\begin{proof} Put $x = x_0$. We denote by $|\vec w|$ the $l_1$-norm of the vector $\vec w$. 

Consider the Taylor expansion of $f(u, \vec y)$ at a point $a = (0, 0, \vec y) \in K$. By the Taylor formula, there exists a polynomial $T^{2k}_{f, a}(\D u, \D\vec y)$ of degree $\leq 2k$, an open neighborhood $U(f, a) \subset Q$ of the point $a \in Q$, and a positive constant $C = C(U(f, a))$ (depending on the estimates of the order $2k+1$ partial derivatives of $f$ in $U(f, a)$) such that
\begin{eqnarray}\label{eq3.AA}
| f(\D u, \vec y +  \D\vec y) - T^{2k}_{f, a}(\D u, \D\vec y) | < C (|\D u| +|\D\vec y|)^{2k+1}
\end{eqnarray}
for all  $(\D u, \vec y +  \D\vec y) \in U(f, a)$.

Since $f(\a(u, \vec y)) = f((u, \vec y))$, there exists a function $g: \R^n_+ \to \R$ such that $\pi^\ast(g)|_Q$ coincides with $f$: just put $g((\pi(u, \vec y)) =_{\mathsf{def}} f(u, \vec y)$.

Using that $f(\a(u, \vec y)) = f((u, \vec y))$ identically, $T^{2k}_{f, a}(\D u, \D\vec y)$ has terms of even degrees in $\D u$ only. 
We introduce the polynomial $\tilde T^{2k}_{f, a}((\D u)^2, \D\vec y)$ in the variables $(\D u)^2, \vec y$ by the formula $\tilde T^{2k}_{f, a}((\D u)^2, \D\vec y) =_{\mathsf{def}}  T^{2k}_{f, a}(\D u, \D\vec y)$. Then we represent $\tilde T^{2k}_{f, a}(\D x, \D\vec y)$ as a sum of a polynomial  $\tilde T^{k}_{f, a}(\D x, \D\vec y)$ of degree  $\leq k$ in the variables $x, \vec y$ and a polynomial $\tilde R^{ > k}_{f, a}(\D x, \D\vec y)$, comprised of monomials whose degrees exceed $k$.   


Then there exists a positive constant $C'$ such that $$|\tilde R^{ > k}_{f, a}(\D x, \D\vec y)| <  C' (|\D x| +|\D\vec y|)^{k+1}$$ for all 
$(\D x, \vec y+ \D\vec y)$ in a sufficiently small open neighborhood $V(f, a) \subset \R^n$ of $\pi(a)$.

Therefore, in some open neighborhood $W(f, a) \subset \R^n$ of $\pi(a)$,  the inequality (\ref{eq3.AA}) can be rewritten as 
\begin{eqnarray}\label{eq3.BB}
\big | g(\D x, \vec y +  \D\vec y) - \tilde T^{k}_{f, a}(\D x) - \tilde T^{> k}_{f, a}(\D x)\big | \nonumber \\ 
\leq \big | g(\D x, \vec y +  \D\vec y) - \tilde T^{k}_{f, a}(\D x)\big | +  \big |\tilde T^{> k}_{f, a}(\D x)\big | \nonumber \\ 
< C (\sqrt{|\D x|} +|\D\vec y|)^{2k+1} + C' (|\D x| +|\D\vec y|)^{k+1}.
\end{eqnarray}

 Note that the positive function $$\psi(|\D x|, |\D\vec y|) =_{\mathsf{def}} (\sqrt{|\D x|} +|\D\vec y|)^{2k+1} / (|\D x| +|\D\vec y|)^{k+1}$$ is bounded from above in an open neighborhood $\mathcal U = \mathcal U(k)$ of $(0, 0)$ in the plane.

Hence, in the vicinity of $\pi(a) = (0, \vec y)$,  the inequality (\ref{eq3.BB}) transforms into the desired Taylor inequality
$$\big | g(\D x, \vec y +  \D\vec y) - \tilde T^{k}_{f, a}(\D x) \big | < \tilde C (|\D x| +|\D\vec y|)^{k+1},$$
where the constant $\tilde C =_{\mathsf{def}} C \cdot \sup_{\mathcal U} \psi(|\D x|, |\D\vec y|) + C' > 0$. 
Therefore $g \in C^{k}(\R^n_+, \R)$.
\hfill 
\end{proof}

\begin{definition}\label{A-property}{\bf (Property $\mathsf A$)} Let $v$ be a traversing boundary generic vector field on a compact connected smooth manifold with boundary.

We say that $v$ {\sf has property} $\mathsf A$ if each $v$-trajectory is transversal to $\d_1 X$ at \emph{some} point, or has the combinatorial type $\om = (2)$.  \smallskip

This property $\mathsf A$ is equivalent to the requirement
\begin{eqnarray}\label{eq_propertyA} 
X\big(v, (33)_\succeq \cup (4)_\succeq\big) = \emptyset.
\end{eqnarray}

In particular, if each connected component of $\d_1X$ is \emph{concave} or \emph{convex} with respect to the $v$-flow, then property $\mathsf A$ is satisfied. Equivalently, if $\d_3X(v) = \emptyset$, then all the combinatorial types of $v$-trajectories are of the form $(11), (1\underbrace{2 \ldots 2}_{k}1), (2)$, so property $\mathsf A$ is valid. 

\hfill $\diamondsuit$
\end{definition}

In particular, $\mathsf A$ is valid for any gradient vector field of a Morse function $f$ on a closed manifold $M$, being restricted to the compliment $X$ to a disjoint union of sufficiently small \emph{convex} balls, centered on the $f$-critical points.\smallskip

Now, we are in position to prove the main result of this paper, dealing with {\sf the topological rigidity} of boundary value problems for a rather general class of ODEs.

\begin{theorem}\label{th3.HOLO}{\bf (The Holography Theorem)} \hfill\break
Let $X_1, X_2$ be two smooth compact connected $(n +1)$-manifolds with boundary, equipped with traversing and boundary  generic fields  $v_1, v_2$, respectively. 
\begin{itemize}
\item Then any smooth diffeomorphism  $\Phi^\d: \d_1X_1 \to \d_1X_2$ such that $$\Phi^\d \circ C_{v_1} =  C_{v_2} \circ \Phi^\d$$ extends to a homeomorphism $\Phi: X_1 \to X_2$ which maps $v_1$-trajectories to $v_2$-trajectories so that the field-induced orientations of trajectories are preserved. The restriction of $\Phi$ to each trajectory is a smooth diffeomorphism.\smallskip

\item If $v_1$ has the property $\mathsf A$ from Definition \ref{A-property}, then the homeomorphism $\Phi$ is a smooth diffeomorphism. In particular, this is the case for any concave vector field $v_1$.\smallskip

\item In general, the conjugating homeomorphism $\Phi: X_1 \to X_2$ is a smooth diffeomorphism outside the closed subsets  $$X_1\big(v_1, (33)_\succeq \cup (4)_\succeq\big) \subset X_1\;\text{ and } \; X_2\big(v_2, (33)_\succeq\cup (4)_\succeq\big) \subset X_2.$$  
If the fields are traversally generic, then the set $X_i(v_i, (33)_\succeq )$ is of codimension $4$. The set $X_i(v_i, (4)_\succeq )$ is of codimension $3$.
\end{itemize}
\end{theorem}

\begin{proof} We divide the proof into three steps.\smallskip

{\bf (1)} First, using that  $\Phi^\d: \d X_1 \to \d X_2$ is a homeomorphism that commutes with the causality  maps $C_{v_i}$,  we see that $\Phi^\d$ gives rise to a well-defined homeomorphism $\Phi^\mathcal T: \mathcal T(v_1) \to \mathcal T(v_2)$ of the trajectory spaces. We claim that $\Phi^\mathcal T$ is a homeomorphism of $\mathbf\Omega^\bullet$-stratified spaces. \smallskip

When the arguments apply to both $v_1$ and $v_2$, in order to simplify the notations, we put $v =_{\mathsf{def}} v_i$ and $X =_{\mathsf{def}} X_i$, where $i =1, 2$.\smallskip

Let $\Gamma: X \to \mathcal T(v)$ be the obvious surjective map. Since $v$ is traversing field, any trajectory reaches the boundary; so the obvious map $\Gamma^\d: \d_1 X \to \mathcal T(v)$ is onto as well. 

Evidently, the fiber of $\Gamma^\d$ consists of the maximal chain of points $x_1 \leadsto x_2 \leadsto \dots  \leadsto x_q$ from  $\d_1 X$ such that $C_{v}(x_j) = x_{j+1}$ for all $j \in [1, q -1]$. By the definition of $C_{v}$, such a chain is exactly the ordered finite locus $\g_{x_1} \cap \d_1 X$. \smallskip

We claim that the combinatorial type $\omega = \omega(\g) \in \mathbf\Omega^\bullet$ of each $v$-trajectory $\g \subset X$ can be recovered from the causality  map $C_{v}: \d_1^+X \to \d_1^-X$ in the vicinity of $\g \cap \d_1 X$. 

For each point $y \in \d_1X$, its multiplicity $m(y) $ with respect to a boundary generic flow $v$ can be detected by the unique pure stratum $\d_j X^\circ := \d_j X^\circ(v)$, $j = m(y)$, to which $y$ belongs. On the other hand, it  can be also detected in terms of the causality  map $C_v$ and its iterations, restricted to the vicinity of $y$. 

Let us justify this observation. Recall that, for boundary generic vector fields, Lemma \ref{local_boundary_generic}  provides us with a model for the divisors $\{D_{\hat\g}\}_{\hat\g}$, localized to a sufficiently small  neighborhood  $U_y$ of $y$ (the set $\hat\g \cap \d X \cap U_y$ is the support of $D_{\hat\g}|_{U_y}$). We choose the $\hat v$-flow adjusted neighborhood $U_y$ with some care: first we chose a small smooth transversal section $S \subset \hat X$ of the $\hat v$-flow, which contains $y$, then we consider the union $V_y$ of $\hat v$-trajectories through the points of $S$, and finally we let $U_y = V_y \cap X$. 

In the $\hat v$-flow adjusted coordinates $(u, x)$, $\d_1 X$ is given by an equation $\{F(u, x) = 0\}$, where a smooth function $F$ has $0$ for its regular value. Then the $\hat v$-trajectory $\hat\g$ is given by the equation $\{x = 0\}$ and the $v$-trajectory $\g$  by $\{F(u, x) \leq 0,\; x = 0\}$. \smallskip

Since $v$ is boundary generic, each  point $y = (u_\star, 0) \in \g \cap \d X$ has multiplicity $m(y) \leq \dim(X)$. So $F(u, 0)$ has a zero at $u_\star$ of multiplicity $m(y) \leq \dim(X)$. By the Taylor formula, this implies that any smooth function $g(u)$ that is $C^\infty$-close to $F(u, 0)$ has finitely many zeros of finite multiplicities, which are localized to the vicinity of the zero set $\{F(u, 0) = 0\}$. Moreover, in the vicinity of $u_\star$, $g(u) = P(u)\cdot Q(u)$, where $P(u)$ is a real polynomial of degree $m(y)$ and $Q(u) > 0$. Therefore any such function $g(u)$ is of the form $\tilde P(u) \cdot \tilde Q(u)$, where $\tilde P(u)$ is a real polynomial of the degree $|\omega| =  \sum_{y \in \g \cap \d X} m(y)$ and $\tilde Q > 0$. Thus, for any trajectory $\g' = \{x = x'\}$ in the vicinity of $\g$ (for any $x'$ sufficiently close to $0$), the intersection $\g' \cap \d_1 X$ is given by the equation, $\{\g_{x'}(u) =_{\mathsf{def}}F(u, x') = 0\}$, and the zero divisor $D_{\g'}$, associated with $\g'$, coincides with the zero divisor $D_\R(\tilde P)$ of a real polynomial $\tilde P$ of degree $|\omega|$ (note that $\deg(D_\R(\tilde P)) \equiv |\omega| \mod 2$).
\smallskip 

Using these local models, the maximal length of a chain 
$$z_1\leadsto  z_2 =_{\mathsf{def}} C_v(z_1) \leadsto z_3 =_{\mathsf{def}} C_v(z_2) \leadsto \dots $$
in any sufficiently small $v$-adjusted neighborhood $U_y \subset \d X$ of $y$ is $\lceil m(y)/2\rceil$, where $\lceil \sim \rceil$ denotes the integral part of a positive number. 
Indeed, if $m(y)$ is even, then the maximal number of roots of \emph{even} multiplicity for a polynomial of degree $m(y)$ is $m(y)/2$, and by Lemma 3.1 \cite{K2}, such $u$-polynomials $g_{x'}(u)$ of the form $\prod_{i =1}^{m(y)/2} (u - u_\star - \e_i)^2$, where all $\{\e_i\}_i$ are distinct, are present in an arbitrary small neighborhood of the polynomial $(u - u_\star)^{m(y)}$ in the coefficient space. 

When $m(y)$ is odd, then the maximal length of a chain $z_1 \leadsto z_2 \leadsto \dots $ in the vicinity of $y$ in $\d_1 X$ is $(m(y) - 1)/2 = \lceil m(y)/2\rceil$. It corresponds either to the $m(y)$-polynomials with one simple root, followed by the maximal number of multiplicity $2$ roots,  or to the $m(y)$-polynomials with the maximal number of multiplicity $2$ roots, followed by a simple root.
 \smallskip

Evidently, the order in which the points $\g \cap \d_1 X$ appear along each trajectory $\g$ is also determined by $C_v$. So the combinatorial type $\omega(\g) \in \mathbf\Omega^\bullet$ of each $v$-trajectory $\g \subset X$ can be recovered from the causality  map $C_v: \d_1^+X \to \d_1^-X$ and its partially-defined iterations. As a result, the information encoded  in $C_v$ is  sufficient for reconstructing the  $\mathbf \Omega^\bullet$-stratified space $\mathcal T(v)$, the image of a finitely ramified map $\Gamma^\d: \d_1 X \to \mathcal T(v)$. \smallskip

Recall that, for \emph{traversally generic} vector fields $v$, the combinatorial type $\omega$ of any trajectory $\g$ determines the $\mathbf\Omega^\bullet$-stratified topology of the germ of $\mathcal T(v)$ at $\g$ (\cite{K3}, Theorem 5.2); in contrast, for just traversing and boundary generic $v$, this determination by $\omega$ alone fails miserably.
\smallskip

So the diffeomorphism $\Phi^\d: \d_1 X_1 \to \d_1 X_2$, which commutes with the causality  maps $C_{v_1}$ and $C_{v_2}$, must take any chain of points 
$$z_1\leadsto  z_2 = C_{v_1}(z_1) \leadsto z_3 = C_{v_1}(z_2) \leadsto \dots$$ in $\d_1 X_1$ to a similar chain in $\d_1 X_2$ with the same multiplicity pattern. \smallskip
Therefore, any smooth diffeomorphism $\Phi^\d$, which commutes with the causality maps, gives rise to a homeomorphism $\Phi^{\mathcal T}: \mathcal T(v_1) \to \mathcal T(v_2)$ which preserves the $\mathbf\Omega^\bullet$-stratifications of the two spaces. Recall that the topology in $\mathcal T(v_i)$ is defined to be the weakest topology for which the obvious map $X_i \to \mathcal T(v_i)$ is continuous. 

Since the stratifications $\{\d_jX_i(v_i)\}_j$  can be recovered from the causality  maps $C_{v_i}$, we get $\Phi^\d(\d_jX_1(v_1)) = \d_jX_2(v_2)$ for all $j > 0$. 
\smallskip

{\bf(2)} Our next goal is to lift  $\Phi^{\mathcal T}$ to a desired homeomorphism (diffeomorphism) $\Phi: X_1 \to X_2$. 

Since $v_2$ is traversing vector field, by Lemma 4.1 from \cite{K1} (or Lemma 5.6 from \cite{K7}), there exists a smooth Lyapunov function $f_2: X_2 \to \R$ such that $\mathcal L_{v_2}(f_2) > 0$ everywhere in $X_2$. 
We use $f_2$ to form an auxiliary function $$f_1^\d =_{\mathsf{def}} (\Phi^\d)^\ast(f_2) : \d_1 X_1 \to \R.$$ 

Since $\Phi^\d$ commutes with the causality maps $C_{v_1}$ and $C_{v_2}$, we conclude that $\Phi^\d$ maps each $v_1$-ordered finite set $\g \cap \d_1X_1$ to $v_2$-ordered set $\Phi^{\mathcal T}(\g) \cap \d_1X_2$.  Therefore if $f_2(x) < f_2(y)$ for some 
$x, y \in \Phi^{\mathcal T}(\g) \cap \d_1X_2$, then $f_1^\d((\Phi^\d)^{-1}(x)) < f_1^\d((\Phi^\d)^{-1}(y))$. As a result, $f_1^\d: \d_1X_1 \to \R$ satisfies the hypothesis of Corollary \ref{cor4.2}. Applying this corollary, we produce a smooth function $f_1: X_1 \to \R$ which extends $f_1^\d$ and has the property $df_1(v_1) > 0$ everywhere in $X_1$.
\smallskip 

With the Lyapunov function $f_1 : X_1 \to \R$ for $v_1$ in place, we are ready to define the homeomorphism $\Phi: X_1 \to X_2$ that extends $\Phi^\d$.  It  takes a typical $v_1$-trajectory $\g \subset X_1$ to the $v_2$-trajectory $\g' \subset X_2$ that projects, with the help of $\Gamma_2$, to the point $\Phi^{\mathcal T}(\g) \in \mathcal T(v_2)$. The restriction of $\Phi$ to each trajectory $\g$  is given by the formula 
\begin{eqnarray}\label{eq12}
\phi_\g^{12}(x) = _{\mathsf{def}} \big(f_2|_{\Gamma_2^{-1}(\Phi^{\mathcal T}(\g))}\big)^{-1} \circ (f_1|_\g).
\end{eqnarray} 
(\ref{eq12}) makes sense since, thanks to the property $f_1^\d = (\Phi^\d)^\ast(f_2^\d)$, the ranges of $f_1: \g \to \R$ and $f_2: \Gamma_2^{-1}\big(\Phi^{\mathcal T}(\g)\big) \to \R$ coincide and the two functions deliver diffeomorphisms between their domains and ranges. In (\ref{eq12}), as usual, we abuse notations: ``$\g$" stands for both a $v$-trajectory in $X$ and for the corresponding point $[\g] := \Gamma(\g)$ in the trajectory space $\mathcal T(v)$. \smallskip

Now we introduce the desired $1$-to-$1$ continuous map $\Phi: X_1 \to X_2$ by the formula $\Phi(x) =_{\mathsf{def}} x'$, where $x'$ belongs to the $v_2$-trajectory over the point $\Phi^{\mathcal T}(\g_x) \in \mathcal T(v_2)$ such that $\phi_{\g_x}^{12}(x) = x'$. By the very construction of the function $f_1: X_1 \to \R$, we get $\Phi |_{\d X_1} = \Phi^\d$. \smallskip

Evidently, $\Phi$ is a homeomorphism since $\Phi^{\mathcal T}$ is a homeomorphism, distinct $v_1$-trajectories are mapped to distinct $v_2$-trajectories, and the restriction of $\Phi$ to each $v_1$-trajectory is a homeomorphism. Moreover, the restriction of $\Phi$ to each trajectory is a smooth orientation preserving diffeomorphism. Similarly, $\Phi^{-1}$ has these properties as well.\smallskip

\begin{figure}[ht]
\centerline{\includegraphics[height=2.3in,width=3.5in]{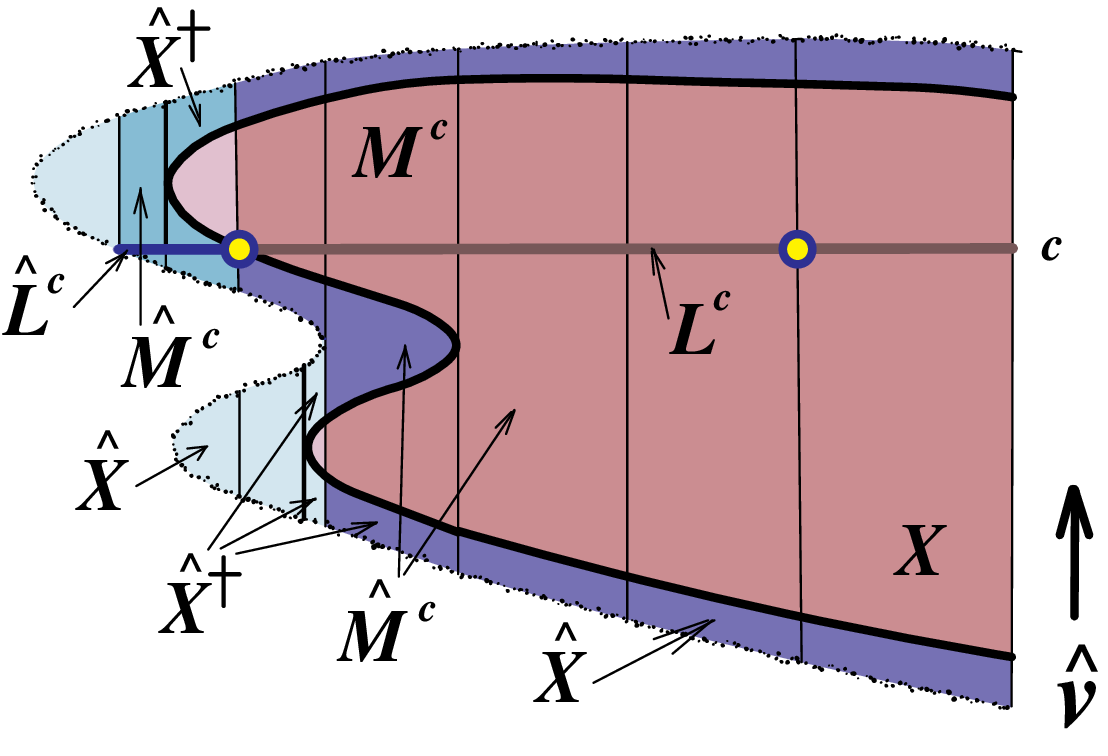}}
\bigskip
\caption{\small{Transversal foliations $\mathcal F(\hat v), \mathcal G(\hat f)$ in $\hat X$,  $\mathcal F(v), \mathcal G(f)$ in $X$, and various loci $\hat L^c, L^c, \hat M^c, M^c$, $\hat X^\dagger$, relevant to the arguments below.}} 
\end{figure}
\smallskip
 
{\bf(3)} In fact, thanks to the smooth dependence of solutions of a non-singular ODE on its initial values, in the cases described by the property $\mathsf A$ (perhaps, always, if Conjecture \ref{conj3.?} is true), $\Phi$ is a diffeomorphism. 

To validate this claim, as usually, we embed $X_i$ properly in a larger open manifold $\hat X_i$ and extend $v_i$ to vector field $\hat v_i$ on $\hat X_i$ so that $d\hat f_i (\hat v_i) > 0$ for an appropriate smooth function $\hat f_i: \hat X_i \to \R$ which extends $f_i$.  We denote by $\mathcal F(\hat v_i)$ the corresponding smooth oriented $1$-dimensional foliation on $\hat X_i$. It is transversal to the smooth $n$-dimensional foliation $\mathcal G(\hat f_i)$, defined by the constant level hypersurfaces $\{\hat f_i^{-1}(c)\}_{c \in \R}$. Let $\hat L^c_i =_{\mathsf{def}} \hat f_i^{-1}(c) \subset \hat X_i$ denote a typical smooth leaf of $\mathcal G(\hat f_i)$. Note that when $c$ is a critical value of $f_i|_{\d_1X_i}$, the locus $f_i^{-1}(c)$ my not be a smooth hypersurface in $X_i$.

The open sets $\{\hat M^c_i =_{\mathsf{def}} \bigcup_{x\in \hat L^c}\hat\g_x\}_{c \in \R}$ cover $X_i$ and thus $\{\d \hat M^c_i =_{\mathsf{def}}\hat M^c_i \cap \d_1 X_i\}_{c \in \R}$ is an open cover of $\d_1 X_i$. Put $M^c_i =_{\mathsf{def}} \bigcup_{x\in L^c}\hat\g_x$.

Finally, we introduce the set $\hat X^\dagger_i$ as the union of all $\hat v_i$-trajectories through $\d_1 X_i$. So $\hat X^\dagger_i$ is a closed subset of $\hat X_i$ and contains $X_i$. Fig. 6 shows the relevant loci.\smallskip

By a construction, similar to the one of $\Phi$, the diffeomorphism  $\Phi^\d$ extends to a homeomorphism $\Phi^\dagger: \hat X_1^\dagger \to \hat X_2^\dagger$. Indeed, each trajectory $\hat\g \subset \hat X_i^\dagger$ is determined by a point $z \in \d X_i$. Let $\hat\g =_{\mathsf{def}} \hat\g_z$. If a leaf $\hat L^c_i$ hits $\hat\g_z$, then the intersection $\hat L^c_i \cap \hat\g_z$ is a singleton. So we may define $\Phi^\dagger$ by the formula $\Phi^\dagger(x) =_{\mathsf{def}} \hat\g_{\Phi^\d(z)} \cap \hat L^c_2$, where $x \in \hat X^\dagger_1$, $c = \hat f_1(x)$, and $z \in \hat\g_x \cap \d_1 X_1$. Since $\Phi^\d$ conjugates the two causality map, this definition does not depend on the choice of $z \in \hat\g_x \cap \d_1 X_1$.

If $x \in \hat X_1$ is such that there exists $z \in \hat\g_x \cap \d_1 X_1$ with the multiplicity $m(z)$ of tangency between $\hat\g_z$ and $\d X_1$ being \emph{odd}, then using the local models of boundary generic fields from Lemma \ref{traversally_generic_A} and formula (\ref{eqVERSAL}), we see that any $\hat v_1$-trajectory in the vicinity of $z$ hits $\d_1 X_1$ (since any real polynomial of an odd degree has a real root). Therefore, in the vicinity of such $x$, the homeomorphism $\Phi^\dagger$ extends further to a homeomorhism $\hat \Phi: \hat X_1 \to \hat X_2$. Since each $v_1$-trajectory, but a singleton, is bounded by two points of an \emph{odd} multiplicity, the only exceptions are the cases when $\hat\g_x \cap \d X_1$ is a singleton of an \emph{even} multiplicity $m(x)$; in the vicinity of such $x$, $\hat X_1$ and $X_1^\dagger$ differ. For these $x$'s, we need an additional reasoning for the existence of an extension of $\Phi^\dagger$ to a germ-homeomorphism $\hat \Phi: \hat X_1 \to \hat X_2$ that maps $\hat v_1$-trajectories to $\hat v_2$-trajectories. It is also based on the local models of boundary generic vector fields from Lemma~\ref{traversally_generic_A}. In fact, Lemma \ref{lem3.A}  provides this reasoning for the points $z \in \d_2^-X_1(v_1) \setminus \d_3X_1(v_1)$, where the field $v_1$ is strictly convex.
\smallskip

By the construction of $\hat\Phi$, we get: (i)  $\hat\Phi^\ast(\hat f_2) = \hat f_1$, and (ii) $\hat\Phi(\hat\g)$ is a leaf of $\mathcal F(\hat v_2)$ for any $\hat v_1$-trajectory $\hat\g$. Thus  $\hat\Phi(\hat L^c_1) = \hat L^c_2$ and $\hat \Phi(\hat M^c_1) = \hat M^c_2$ for any $c \in \R$. \smallskip  

Given two smooth manifolds $Y_1$ and $Y_2$, a map $\Psi: Y_1 \to Y_2$ is smooth if and only if its composition with each local coordinate in $Y_2$ is a smooth function in the local coordinates on $Y_1$.  

The leaves of the smooth foliations $\mathcal F(\hat v_i)$ and $\mathcal G(\hat f_i)$ can be locally defined by freezing complementary groups of the appropriate smooth local coordinates in $\hat X_i$. Recall that $\hat\Phi$ maps the smooth foliation $\mathcal F(\hat v_1)$ to the smooth foliation $\mathcal F(\hat v_2)$, the restriction of $\hat\Phi$ to the the leaves-trajectories being a smooth diffeomorphism.  Since $\hat\Phi$ also maps the smooth foliation $\mathcal G(\hat v_1)$ to the smooth foliation $\mathcal G(\hat v_2)$, if the restrictions $\{\hat\Phi: \hat L_1^c \to \hat L_2^c\}_{c \in \R}$ of $\hat\Phi$ to the leaves of $\mathcal G(\hat v_1)$ are smooth maps, we may conclude that the homeomorphism $\Phi: X_1 \to X_2$ is a smooth map.\smallskip 
 
Since $\Phi^\d$ is a smooth diffeomorphism, the image $\Phi^\d(z) \in \d_1 X_2$ depends smoothly on $z \in \d_1X_1$.  Therefore, the image point $\Phi(x) \in X_2$ depends smoothly on a point $z \in  \g_x \cap \d M^c_1$, where $c = f_1(x)$ (as long as $\g_z \cap L^c_1 \neq \emptyset$).

A priori, this does not imply that $\Phi(x)$ depends smoothly on $x$! For this assertion to be valid, it would be sufficient to validate Conjecture \ref{conj3.?}. 
\smallskip

However, as we will see now, when the property $\mathsf A$ is available, we can overcome this difficulty. When the $\hat v_1$-trajectory $\hat\g$ through a point $x \in X_1$ is transversal to $\d_1 X_1$ at \emph{some} point $z \in \d_1 X_1$, then, in the vicinity of $x$, the $\hat v_1$-induced map $p_1^\d: \d\hat M^c_1 \to \hat L^c_1$, $c = f_1(x)$, admits a \emph{smooth} local section $\s_1: \hat L^c_1 \to \d \hat M^c_1$ which is transversal to the fibers of $p_1: \hat M^c_1 \to \hat f_1^{-1}(c)$. That section is delivered by the boundary $\d_1 X_1$ in the vicinity of $z$. In such a case, $\Phi$ is smooth in the vicinity of $x$, since the composition $p_2^\d \circ \Phi^\d \circ \s_1: \hat L^c_1 \to \hat L^c_2$ is a smooth map. This conclusion applies to all $v_1$-trajectories $\g$ that are bounded by at least one point of multiplicity $1$. The exceptions are the trajectories bounded by \emph{two} points of odd multiplicities that exceed $1$, that is, by the trajectories whose combinatorial type belongs to the poset  $(33)_\succeq \subset \mathbf \Omega^\bullet$. Other exceptions to the transversality case may occur for the trajectories whose combinatorial types belong to the poset $(4)_\succeq \subset \mathbf \Omega^\bullet$. They include all the combinatorial types $(2k)$, where $k \geq 4$. \smallskip

In the special case of trajectories of the combinatorial type $(2) \in \mathbf \Omega^\bullet$, the local differentiability of $\Phi$ in the vicinity of $z \in \d_2^-X_1(v_1)\setminus \d_3X_1(v_1)$ follows from Lemma \ref{lem3.A}.  Indeed, in the special smooth coordinates $(u, x_0, \vec y)$, where $\vec y =(y_1, \dots , y_{n-1})$, in the vicinity of such point $z$, the boundary $\d_1 X_1$ is given by an equation $\{u^2 + x_0 = 0\}$, while $X_1$ by the inequality $\{u^2 + x_0 \geq 0\}$. Each $\hat v_1$-trajectory is specified by freezing the coordinates $(x_0, \vec y)$. The smooth hypersurfaces $\{\hat L_1^c\}$ are transversal to the $\hat v_1$-trajectories. Since $\hat\Phi^\d$  maps $\d_2X_1(v_1) \setminus \d_3X_1(v_1)$ to $\d_2X_2(v_2) \setminus \d_3X_2(v_2)$, a similar system of smooth coordinates is available in the vicinity of $\Phi^\d(z)$. We use the symbol `` ' "  to denote them. 

By the previous transversality argument, the homeomorphism $\Phi$ may fail to be a local diffeomorphism at the points of the locus $\d_2X_1(v_1)$; so we need to investigate whether $\Phi$ is differentiable in the vicinity of $\d^-_2X_1(v_1)$. 

The following  arguments are based on Lemma \ref{lem3.A} and uses its notations.
In the appropriate local coordinates $(u, x_0, \vec y)$ and $(u', x'_0, \vec y')$, the smooth diffeomorphism $\Phi^\d: Q \to Q'$ of two quadratic hypersurfaces maps the $p_1$-folding locus $K$ to the $p_2$-folding locus $K'$ and commutes with the two causality maps $\a: Q \to Q$ and $\a': Q' \to Q'$.

The local coordinate function $x'_0: X_2 \to \R$ pulls back to a smooth $\a'$-invariant function $p_2^\ast(x'_0): Q' \to \R$. Since $\Phi^\d$ is a smooth diffeomorphism, the further pull-back  $\phi = (\Phi^\d)^\ast(p_2^\ast(x'_0)): Q \to \R$ is a smooth function on $Q$. Because $\Phi^\d$ commutes with $\a$ and $\a'$, $\phi$ is $\a$-invariant.  
Therefore, by Lemma \ref{lem3.A},  $\phi$ is a restriction to $Q$ of a smooth $u$-independent function $\chi$ in the variables $x_0, \vec y$.

Similarly, using that $\Phi^\d$ commutes with $\a$ and $\a'$, we conclude that $\psi = (\Phi^\d \circ p_2)^\ast(\vec y'): Q \to \R^{n-1}$  is a smooth and $\a$-invariant map. Therefore $\psi$ is a restriction to $Q$ of a smooth map $\theta: X_1 \to \R^{n-1}$ that depends only on the coordinates $(x_0, \vec y)$. 

At the same time, $\hat\Phi^\ast(\hat f_2) = \hat f_1$. The functions $(\hat f_2, x'_0, \vec y')$ form a smooth local system of coordinates. By the arguments above, the pull-back under $\hat\Phi$ of these coordinates are smooth on $\hat X_1$. Therefore, $\Phi$ is a smooth homeomorphism in the vicinity of $K$. By the same token, exchanging the roles of $X_1$ and $X_2$, $\Phi^{-1}$ is smooth as well. 

This concludes the proof of Theorem \ref{th3.HOLO}.
\hfill 
\end{proof}

\noindent{\bf Remark 3.5.} Let $v$ be a traversing boundary generic vector field on $X$. Among other things, Theorem \ref{th3.HOLO}  claims that \emph{any} diffeomorphism of the boundary $\d_1 X$, which commutes with the (partially defined) causality map $C_v$, extends to a homeomorphism (when $v_1$ satisfies $\mathsf A$, to a smooth diffeomorphism) of $X$! \hfill $\diamondsuit$ 
\smallskip

\begin{corollary}\label{cor3.8} Let $X_1, X_2$ be two smooth compact connected $(n +1)$-manifolds with boundary, equipped with traversing and boundary  generic fields  $v_1, v_2$, respectively. Then any diffeomorphism  $\Phi^\d: \d X_1 \to \d X_2$ such that $$\Phi^\d \circ C_{v_1} =  C_{v_2} \circ \Phi^\d$$
generates a stratification-preserving homeomorphism $\Phi^\mathcal T: \mathcal T(v_1) \to \mathcal T(v_2)$ of the corresponding $\mathbf\Omega^{\bullet}$-stratified trajectory spaces. 

If $v_1$ has property $\mathsf A$ from Definition \ref{A-property}, then $\Phi^\mathcal T$ induces an isomorphism $$(\Phi^\mathcal T)^\ast : C^\infty(\mathcal T(v_2)) \to C^\infty(\mathcal T(v_2))$$ of the algebras of smooth functions on the two trajectory spaces---the two spaces are ``diffeomorphic". 
\end{corollary}

\begin{proof} By the proof of Theorem \ref{th3.HOLO}, there exists a diffeomorphism $\Phi : X_1 \to X_2$ which takes $v_1$-trajectories to $v_2$-trajectories and extends $\Phi^\d$, while preserving their combinatorial tangency patterns. Therefore, $\Phi$ maps every  smooth function $f: X_2 \to \R$  that is constant on each $v_2$-trajectory to a continuous function $f\circ \Phi: X_1 \to \R$  that is constant on each $v_1$-trajectory. When $X_1(v_1, (33)_\succeq \cup (4)_\succeq) = \emptyset$, then $\Phi$ is a smooth diffeomorphism; so $\Phi^\ast(f)$ is a smooth function also constant along the $v_1$-trajectories.

Similar argument applies to the inverse homeomorphism/ diffeomorphism $(\Phi)^{-1}: X_2 \to X_1$. 
\hfill 
\end{proof}
\smallskip

Theorem \ref{th3.HOLO} has another ``holographic" implication:

\begin{corollary}\label{cor3.9} For a boundary generic and traversing vector field $v$ on $X$, the topological type of the pair $(X, \mathcal F(v))$ can be recovered from each of the following structures on its boundary $\d_1X$:
\begin{enumerate} 
\item the causality  map $C_v: \d_1^+X(v) \to \d_1^-X(v)$,
\item the poset $(\mathcal C^\d(v), \succ)$ whose elements are the points of $\d_1X$,
\item the category $\mathbf{\mathcal Cat}^\d(v)$, determined by the poset $(\mathcal C^\d(v), \succ)$.
\end{enumerate}

\begin{itemize} 
\item When $v$ has property $\mathsf A$ from Definition \ref{A-property}, the above homeomorphism is a smooth diffeomorphism.\smallskip

\item As a result, all the topological invariants of $X$ (such as rational Pontryagin classes of $X$) can be recovered from each of the three previous structures on $\d_1 X$.  \smallskip
 
\item When $v$ has property $\mathsf A$, all the invariants of the smooth structure on $X$ (such as all the characteristic classes of the tangent bundle $\tau(X)$) can be recovered from each of the three previous structures on $\d_1 X$.  
\end{itemize}
\end{corollary}

\begin{proof} Consider two manifolds $X_1$ and $X_2$ which carry traversing boundary  generic vector fields $v_1$ and $v_2$. Assume that the two manifolds share a common boundary: $\d_1X_1= \d_1X_2$. If the two fields induce identical causality maps, then, according to Theorem \ref{th3.HOLO}, the diffeomorphism $\Phi^\d:= id_{\d_1 X}$, extends to a homeomorphism $\Phi: X_1 \to X_2$ so that the oriented foliation $\mathcal F(v_1)$ is mapped to the oriented foliation $\mathcal F(v_2)$, the homeomorphism $\Phi$ being a diffeomorphism on each leaf. 

The equivalence of the three structures in the statement of the corollary has been established in the discussion that has followed formula (\ref{eq4.1}). 

When  $v_1$ has property $\mathsf A$, the homeomorphism $\Phi$ may be assumed to be a smooth diffeomorphism. 
\hfill 
\end{proof}


\noindent{\bf Example 4.1.} The statement of Corollary \ref{cor3.9} is not obvious even for the nonsingular gradient flows on $2$-dimensional manifolds. Consider a compact surface $X$ with a connected boundary $\d_1 X \approx S^1$ and a traversally  generic field $v$ on $X$. Then $\d_1^+X$ is a disjoint union of $q$ arcs in $S^1$. The set $\d_1^-X$ is a disjoint union of equal number of arcs. 

The causality map $C_v:  \d_1^+X \to \d_1^-X$ can be represented by a graph $G_v \subset \d_1^+X \times \d_1^-X$, drawn in a set of $q \times q$ of black unitary squares of the $2q \times 2q$ checker board, the sums of indexes of each square in the $2q \times 2q$ table being odd. The graph $G_v$ has a finite number of discontinuity points with well-defined left and right limits for each arc of $G_v$. The interior of each arc of $G_v$ is smooth. 

According to Corollary \ref{cor3.9}, this graph $G_v$ ``knows" everything about the topology of $X$ and the dynamics of the un-parametrized $v$-flow on it, up to a diffeomorphism of $X$!  Even the claim about the topological type of $X$ has some subtlety: according to the Morse formula for vector fields \cite{Mo}, to calculate $\chi(X)$, and thus to determine the topological type of $X$, we need to know not only $\chi(\d_1^+X) = q$ (which we obviously do), but also the integer $\chi(\d_2^+X)$, which can be extracted by iterating the map $C_v$. This presumes that the polarity of each of the $2q$ points from $\d_2X$ can be recovered from $C_v$ or $G_v$. We leave to the reader to discover the recipe.  \hfill $\diamondsuit$  
\smallskip

\noindent{\bf Example 4.2.}  For a transversally generic $v$ on a smooth $4$-dimensional $X$, the locus $X(v, (33)_\succeq) = \emptyset$ for dimensional reasons. Since  $X(v, (4)_\succeq)$ is a finite set residing in $\d_1 X$, we conclude that all the Gauge invariants of compact smooth $4$-manifolds $X$ with boundary can be recovered from the causality map $C_v:  \d_1^+X \to \d_1^-X$. As a practical matter, this recovery must be very challenging...
\hfill $\diamondsuit$
\smallskip

The next theorem suggests that traversing vector fields and their causality maps give rise to \emph{a new representation} of smooth manifolds with the spherical boundary. 

\begin{theorem}\label{th3.4} For $n \geq 3$, any compact connected smooth $(n+1)$-dimensional manifold $X$ with the spherical boundary can be represented, up to a homeomorphism, by a semi-continuous map $C: D^n_+ \to D^n_-$ between a pair of $n$-balls. The $C$-fibers are finite of cardinality $n+1$ at most, and a generic fiber is of cardinality $1$. This map $C$ captures the topological type of $X$. 

For $n=3$, $C$ captures the smooth topological type of the $4$-manifold $X$.
\end{theorem}
\begin{proof} Consider any compact connected smooth manifold $X$ with a spherical boundary $\d_1X = S^n$. By Theorem 3.1 from \cite{K1} and Theorem 3.5 from \cite{K2}, there is an open set $\mathcal D(X)$ of traversally generic vector fields $v$, such that $\d_1^+X$ is diffeomorphic to a ball $D^n_+ \subset S^n$. Then $\d_1^-X$ is the complimentary ball $D^n_-$. According to Corollary \ref{cor3.9}, for any $v \in  \mathcal D(X)$, the topological type of the manifold $X$ is determined by the semi-continuous causality map  $C_v: D^n_+ \to D^n_-$ (equivalently, by its graph $\Gamma(C_v) \subset D^n_+ \times D^n_-$). 

For $n \leq 3$, the locus $X(v, (33)_\succeq) = \emptyset$ and $X(v, (4)_\succeq)$ is a finite set, residing in $\d_1 X$. So by Corollary \ref{cor3.9}, this map $C_v$ captures the smooth topological type of $X$. 
\hfill 
\end{proof}

Compare this description of $X$ as a map $C_v: D^n_+ \to D^n_-$ with the description of the trajectory space $\mathcal T(v)$, given by the Origami Theorem  Theorem 3.1 from \cite{K8}. For a specially designed traversally generic $v$, the Origami Theorem presents the trajectory space as the continuous image of a ball $D^n$, where $n = \dim(\d_1X)$.
\smallskip

\noindent {\bf Example 4.3}
Consider a liquid flow trough a given volume $X$ with a smooth boundary. We assume that the flow velocity $v$ does not vanish in $X$. We think about $\d X$ as the hypersurface, where a multitude of measuring devices are positioned. The basic assumption is that their presence and measuring activity does not alter the flow. 

Any particle which enters the volume is registered, and its next appearance at a point of $\d_1 X$ is registered as well.  According to the Holography Theorem \ref{th3.HOLO}, these data allow for a reconstruction of the bulk $X$ and of the un-parametrized dynamics of the flow in it, up to a homeomorphism (a diffeomorphism) of $X$ which is identity on its boundary. 
\hfill $\diamondsuit$  
\smallskip

Now consider any time-dependent vector field $u(t)$, $t \in \R$, on a $n$-dimensional manifold $Y$ \emph{without} boundary. Then $u(t)$ gives rise to a non-vanishing vector field $v =_{\mathsf{def}} (u(t), 1)$ on the manifold $Y \times \R$. Note that $v$ is a gradient-like field with respect to the function $T(y, t) = t$ on $Y \times \R$. We call a pair $(y, t)$ an {\sf event} since we think of $T$ as time, and of $Y$ as space.

Let $X \subset Y \times \R$ be a $0$-dimensional compact submanifold with a smooth boundary.  Since  $dT(v) = 1$, any $v$-trajectory $\g(t)$ that passes through a point of $X$ is contained in $X$ for a compact set of instances $t \in \R$.  

Assume that  $X \subset Y \times \R$ is such that $v$ is boundary generic with respect to the boundary $\d_1 X$. In view of Theorem 3.5 from \cite{K2}, this assumption can be satisfied by a small perturbation $\tilde v$ of $v$. In fact, such perturbation $\tilde v$ can be of the form $( \hat u(t, y), 1)$ since the property of a field to be boundary generic depends only on its direction, and not on its magnitude. \smallskip

Let us call $X$ the ``{\sf event manifold}" and its boundary $\d_1 X$ the ``{\sf event horizon}". Note that the event manifold is chosen as independent set of data, not directly related to the time-dependent dynamic system $u(t)$ on the manifold $Y$.  
We call the events in $X$ {\sf internal} and in $Y \times \R \setminus X$ {\sf external}.
\smallskip

Thus $u(t)$ defines the causality  map $C_v: \d_1^+X(v) \to \d_1^-X(v)$ which takes  each  ``entrance" point $x_0 = (y_0, t_0)$ on the event horizon $\d_1 X$ to the closest along the $v$-trajectory trough $x_0$  ``exit" point $x_1 = (y_1, t_1)$ on $\d_1 X$. 

We can think of the event $x_0$ as the {\sf cause} of the event  $x_1$, so that $C_v$ indeed becomes the {\sf causality map} or the {\sf causality relation} on the horizon $\d_1 X$.  
\smallskip

The Holography Theorem \ref{th3.HOLO} and Corollary \ref{cor3.9} have the following important interpretation which applies to time-dependent vector fields: 

\begin{theorem}\label{th3.6}{\bf (The Causal Holography Principle)}\hfill\break 
Let $u(t)$, $t \in \R$, be a time-dependent smooth vector field on a $n$-dimensional smooth manifold $Y$ without boundary. 

For any compact $(n+1)$-dimensional smooth event manifold $X \subset Y \times \R$ such that the field $v = (u, 1)$ is boundary generic with respect to  $\d_1 X$, the causality relation on the event horizon $\d_1 X$ determines the pair $(X, \mathcal F(v))$, up to a homeomorphism  of $X$ which is the identity on the event horizon. \smallskip

When $v$ has property $\mathsf A$ from Definition \ref{A-property}, then the causality relation on the event horizon $\d_1 X$ determines the pair $(X, \mathcal F(v))$, up to a smooth diffeomorphism of $X$ which is the identity on the event horizon. 
\hfill $\diamondsuit$  
\end{theorem}

\noindent{\bf Remark 3.6.}
We do not claim that the reconstruction of the event manifold $X$ from the causality map also allows for the reconstruction of its slicing by the fixed-time frames! \hfill $\diamondsuit$
\smallskip

In turn, Theorem \ref{th3.6} has has the following interpretation:

\begin{corollary}\label{cor4.5}{\bf(The topological rigidity of continuations for ODEs)}\hfill\break 
\noindent Let $Y$ be a smooth $n$-manifold without boundary and $X \subset Y \times \R$ a compact smooth submanifold of dimension $n+1$.  Let $u_1(t), u_2(t)$, $t \in \R$, be two time-dependent smooth vector fields on $Y$ such that $u_1(y, t) = u_2(y, t)$ for all ``external" events $(y, t) \in (Y \times \R) \setminus X$. Assume that $v_1 = (u_1, 1)$ and $v_2 = (u_2, 1)$ are boundary generic fields on $X$. Suppose  that the two causality maps, $C_{v_1}: \d_1^+X \to \d_1^-X$ and $C_{v_2}: \d_1^+X \to \d_1^-X$ are  identical. \smallskip

Then the two dynamical systems, generated by $v_1$ and $v_2$ on $Y \times \R$, are topologically equivalent via a homeomorphism 
which is the identity on the event horizon. \smallskip 

When $v_1$ has property $\mathsf A$ (in particular, when the field $v_1$ is concave or convex with respect to $\d_1 X$), then the two dynamical systems are equivalent via a smooth diffomorphism which is the identity on the event horizon.
\hfill $\diamondsuit$
\end{corollary}

In search for further applications of the Holography Theorem \ref{th3.HOLO}, let us let us pay a brief visit to the Classical Hamiltonian/Lagrangian Mechanics.\smallskip  

Let $TM$ be a tangent bundle of a $n$-dimensional smooth manifold $M$ without boundary, and $q_1, \ldots, q_n,\, \dot q_1, \ldots, \dot q_n$ local coordinates in $TM$. In these coordinates $(q, \dot q)$, the Lagrange function $L: TM \times \R \to \R$ may be written as $L(q, \dot q, t)$. The Euler-Lagrange equations $\big\{\frac{d}{dt} \frac{\d L}{\d \dot q_i} - \frac{\d L}{\d q_i} = 0\big\}_{I\in [1,n]}$ describe the curve $\g = q(t)$ which minimizes the path integral $\int_{t_0}^{t_1} L \, dt$. The Hamiltonian function $H: T^\ast M \times \R \to \R$ is defined by 
$H(p, q, t) := p\cdot \dot q - L(q, \dot q, t)$ with $p = \frac{\d L}{\d \dot q}$.
In these coordinates, the Euler-Lagrange equations transform into the Hamilton system of ODEs: 
\begin{eqnarray}\label{eq4.H}
\dot q =\;\; \frac{\d H}{\d p},\quad
\dot p = -  \frac{\d H}{\d q}  = \frac{\d L}{\d q},\quad 
\frac{\d H}{\d t} = -\frac{\d L}{\d t}.
\end{eqnarray} 


In the canonical coordinates $(q, p, t)$, we consider the vector field 
\begin{eqnarray}\label{eq4.vH}
v_H =_{\mathsf{def}} (\dot q, \dot p, 1) = \Big(\frac{\d H}{\d p},  -\frac{\d H}{\d q},\, 1\Big),
 \end{eqnarray}
 whose projection on $T^\ast M$ is the time-dependent Hamiltonian vector field.
\smallskip

Applying Theorem \ref{th3.6} to the Hamiltonian system (\ref{eq4.H}), we get the following statement.

\begin{corollary}\label{cor4.6} For a smooth manifold $M$ without boundary, consider the smooth Hamiltonian system (\ref{eq4.H}) on $T^\ast M$.  Assume that: 
\begin{itemize}
\item a number $c$ is a regular value of a smooth function $F: T^\ast M \times \R \to \R$,
\smallskip
 
\item the set $$X =_{\mathsf{def}} \{x \in T^\ast M \times \R| \;  F(x) \leq c\}$$ is compact in $T^\ast M \times \R$, \smallskip

\item the vector field $v_H$ from (\ref{eq4.vH}) is boundary  generic with respect to the event horizon  $$\d_1 X =_{\mathsf{def}} \{x \in T^\ast M \times \R| \;  F(x) = c\}.$$  
\end{itemize}

Then the causality map/relation $C_{v_H}$ on the event horizon $\d_1 X$ allows for a reconstruction of the pair $(X, \mathcal F(v_H))$, up to a  homeomorphism of $X$ which is the identity on $\d_1 X$. 

If $\d_3X(v_H) = \emptyset$, then the reconstruction is possible, up to a smooth diffeomorphism.
\hfill $\diamondsuit$
\end {corollary}

\begin{question} \emph{The main unresolved issue here is:} 
``How abundant are the Hamiltonian systems $\{v_H\}_H$ that are traversing and boundary generic (alternatively, traversally generic) with respect to a given event horizon $\d_1 X \subset T^\ast M \times \R$?" \hfill $\diamondsuit$
\end{question} 

It follows from \cite{K2} that if a Hamiltonian field $v_H$ has this property relative to a given $\d_1 X$, then for any Hamiltonian function $\tilde H$ that is $C^\infty$-close to $H$, the vector field $v_{\tilde H}$ also will be traversing and boundary generic with respect to $\d_1X$.
\smallskip

We  know that any non-vanishing gradient-like field $v$ can be $C^\infty$-approximated by a traversally  generic field on $X$ (Theorem 3.5 from \cite{K2}). So the open question is whether an approximation is possible within the universe of Hamiltonian fields. \hfill

\section{On applications of Holography Theorems to geodesic flows}

In \cite{K6}, 
we apply the Holographic Causality Principle to the geodesic flows on the spaces of unit tangent vectors of compact Riemannian manifolds \emph{with boundary}. Such applications include the {\sf inverse geodesic scattering problems} and the {\sf geodesic billiards}. Let us describe briefly the flavor of these applications.\smallskip

Let $M$ be a compact connected $n$-dimensional smooth Riemannian manifold with boundary, and $g$ a smooth Riemannian metric on $M$. Let $SM \to M$ denote the tangent spherical bundle of $M$. Then the metric $g$ induces a geodesic vector field $v^g$, a non-vanishing section in the tangent bundle $T(SM)$ (for example, see \cite{Be} for the definition and basic properties of geodesic flows). 

\begin{definition}\label{def4.1} Let $M$ be a compact connected $n$-dimensional smooth Riemannian manifold with boundary. We say that a Riemannian metric $g$ on $M$ is {\sf non-trapping} if the geodesic vector field $v^g$ on $T(SM)$ admits a smooth differentiable Lyapunov function $F: SM \to \R$ such that $dF(v^g) > 0$. \hfill $\diamondsuit$
\end{definition}

For a non-trapping $g$, any geodesic curve $\g \subset M$ is an image of a closed segment, or is a singleton. The converse is true as well \cite{K6}.  

For non-trapping metrics $g$, and only for such metrics, the causality map $$C_{v^g}: \d_1^+(SM)(v^g) \to \d_1^-(SM)(v^g)$$ is well-defined. In fact, its domain and range are diffeomorphic via the reflection map. \smallskip

We call $C_{v^g}$ {\sf the scattering map} since it takes any pair $(m, v)$, where $m \in \d M$ and a unitary tangent vector $v \in T_m(M)$ points inside $M$ or is tangent to its boundary $\d M$, to the pair $(m', v')$, where $m' \in \d M$ and $v' \in T_{m'}(M)$ points outside $M$ or is tangent to $\d M$. Here $m' \neq m$ is the first point of $\d M$ that lies on the unique geodesic curve $\g \subset M$ that passes trough $m$ in the direction of $v$, and $v'$ is the velocity vector of $\g$ at $m'$.  If, in the vicinity of $m$, $\g \cap M = m$, then we put $C_{v^g}(m, v) =_{\mathsf{def}} (m, v)$.

\begin{definition}\label{def4.2} We say that a metric $g$ on $M$ is {\sf boundary generic} if the geodesic vector field $v^g$ is boundary generic with respect to $\d_1(SM)$ in the sense of Definition \ref{boundary_generic}. \hfill $\diamondsuit$
\end{definition}

In the space of all Riemannian metrics on $M$, the non-trapping metrics and the boundary generic metrics form  open sets. \smallskip

\begin{definition}\label{def4.3} Given two compact smooth Riemannian $n$-manifolds, $(M_1, g_1)$ and $(M_2, g_2)$, consider the geodesic fields $v^{g_1}$ on $SX_1$ and $v^{g_2}$ on $SX_2$, respectively. They generate the oriented 1-dimensional geodesic foliations $\mathcal F(v^{g_1})$  and $\mathcal F(v^{g_2})$.
\begin{itemize}
\item We say that the metrics $g_1$ and $g_2$ are {\sf geodesically smoothly conjugated} if there is a smooth diffeomorphism  $\Phi: SM_1 \to SM_2$ that maps each leaf of $\mathcal F(v^{g_1})$ to a leaf of $\mathcal F(v^{g_2})$, the orientations of the leaves being preserved \smallskip

\item We say that the metrics $g_1$ and $g_2$ are {\sf geodesically topologically conjugated} if there is a homeomorphism $\Phi: SM_1 \to SM_2$  that  maps each leaf of $\mathcal F(v^{g_1})$ to a leaf of $\mathcal F(v^{g_2})$, the map $\Phi$ on every leaf being an orientation-preserving diffeomorphism. 

\hfill $\diamondsuit$
\end{itemize} 
\end{definition}

Applying Theorem \ref{th3.HOLO}, we get the following theorem \cite{K6}.

\begin{theorem}\label{th4.1}{\bf (The topological rigidity of the geodesic flow for the inverse scattering problem)}\smallskip

Let $(M_1, g_1)$ and  $(M_2, g_2)$ be two smooth compact connected Riemannian $n$-manifolds with boundaries,  and let the metrics $g_1$, $g_2$ be non-trapping and geodesically boundary generic. 
 
Assume that the scattering maps $$C_{v^{g_1}}: \d_1^+(SM_1) \to \d_1^-(SM_1)\;\; \text{and} \;\; C_{v^{g_2}}: \d_1^+(SM_2) \to \d_1^-(SM_2)$$ are conjugated by a smooth diffeomorphism $\Phi^\d: \d_1(SM_1) \to \d_1(SM_2)$.  \smallskip

Then the metrics $g_1$ and $g_2$ are geodesically topologically conjugated.\smallskip 

If each component of the boundary $\d M_1$ is either concave or convex with respect to $g_1$, then the two metrics are geodesically  smoothly conjugated.\hfill $\diamondsuit$
\end{theorem}

\begin{corollary}\label{cor4.1} Assume that a smooth compact connected Riemannian manifold $M$ admits a geodesically boundary generic non-trapping Riemannian metric $g$. Then the scattering  map $C_{v^g}: \d_1^+(SM) \to \d_1^-(SM)$  allows for a reconstruction of the $\mathbf\Omega^\bullet$-stratified topological type
of the space $\mathcal T(v^g)$  of un-parametrized geodesics on $M$. \smallskip

If each component of the boundary $\d M$ is either concave or convex with respect to a non-trapping $g$, then $C_{v^g}$ allows for a reconstruction of the 
$\mathbf\Omega^\bullet$-stratified smooth topological type of $\mathcal T(v^g)$, determined by the algebra $C^\infty(\mathcal T(v^g))$ of smooth $v^g$-invariant functions on $SM$ (see Definition \ref{def2.1}).
\hfill $\diamondsuit$
\end{corollary}

We say that manifolds $M$ and $M'$ share the same  {\sf stable  topological/smooth type}, if $M \times S^{n-1}$ and $M' \times S^{n-1}$ are homeomorphic/ smoothly diffeomorphic.\smallskip

Theorem \ref{th4.1} leads to the following statement \cite{K6}:

\begin{theorem}\label{th4.2} Assume that a compact connected $n$-manifold $M$ with boundary admits a boundary generic non-trapping Riemannian metric $g$.  \smallskip

Then the geodesic scattering map $C_{v^g}: \d_1^+(SM) \to \d_1^-(SM)$ allows for a reconstruction of the cohomology rings $H^\ast(M; \Z)$ and $H^\ast(M, \d M; \Z)$, as well as for a reconstruction of the homotopy groups $\{\pi_i(M)\}_{i < n}$. \smallskip

Moreover, the Gromov simplicial semi-norms $\|\sim\|_{\mathbf \D}$  on $H^\ast(M; \R)$ and on $H^\ast(M, \d M; \R)$ (see \cite{Gr}) can be reconstructed form $C_{v^g}$. In particular,  the simplicial volume $\|[M, \d M] \|_{\mathbf \D}$ of the fundamental cycle $[M, \d M]$ can be recovered form $C_{v^g}$.\smallskip

If, in addition, $M$ has a trivial tangent bundle, then the stable topological type of $M$  is also reconstructable from the geodesic scattering map. \hfill $\diamondsuit$
\end{theorem}

In the spirit of Theorem 1.3 from \cite{BCG}, by combining the Mostov Rigidity Theorem \cite{Most} with Theorems \ref{th4.1} and \ref{th4.2}, in \cite{K6} we get  the following result. It is inspired by the image of geodesic motion of a bouncing particle in the complement $M$ to a number of disjoint balls, placed in a closed hyperbolic manifold $N$, $\dim(N) \geq 3$. The balls are placed so ``dense" in $N$ that every geodesic curve hits some ball. Under these assumptions, the probe particle collisions with the boundary $\d M$ ``feel the shape of $N$".

\begin{theorem}\label{th4.5} Let $n \geq 3$.  Consider two closed smooth locally symmetric Riemannian $n$-manifolds, $(N_1, g_1)$ and $(N_2, g_2)$, with negative sectional curvatures. Let a connected manifold $M_i$ ($i = 1, 2$) be obtained from $N_i$ by removing the interior of a smooth codimension zero submanifold $U_i \subset N_i$, such that the induced homomorphism $\pi_1(M_i) \to \pi_1(N_i)$ of the fundamental groups is an isomorphism\footnote{By a general position argument, this the case when $U_i$ has a spine of codimension $3$ at least. In particular, $U_i$ may be a disjoint union of $n$-balls.}.

Let the restriction of the metric $g_i$ to $M_i$ be boundary generic and non-trapping. Assume also that the two geodesic scattering maps $$C_{v^{g_1}}: \d_1^+(SM_1) \to \d_1^-(SM_1), \quad C_{v^{g_2}}: \d_1^+(SM_2) \to \d_1^-(SM_2)$$ are conjugated via a smooth diffeomorphism $\Phi^\d: \d (SM_1) \to \d (SM_2)$\footnote{Thus the boundaries $\d U_1$ and $\d U_2$ are stably diffeomorphic.}.\smallskip

Then $\Phi^\d$ determines a unique diffeomorphism $\phi: N_1 \to N_2$ such that $\phi^\ast(g_2) = c \cdot g_1$ for a constant $c > 0$.\smallskip
\hfill $\diamondsuit$
\end{theorem}

For non-trapping geodesic flows on Riemmanian manifolds $M$ with boundary, the scattering map $C_{v^g}: \d_1^+SM \to \d_1^-SM$ can be composed with the reflections with respect to $\d M$ (according the law ``the angle of incidence is equal to the angle of reflection") to produce the {\sf billiard map} $B_{v^g}: \d_1^+SM \to \d_1^+SM$ . For $B_{v^g}$, arbitrary iterations are available. The dynamics of $B_{v^g}$-iterations is the subject of  flourishing research. In particular, in \cite{K9}, we analyze some ``holographic" properties of the $B_{v^g}$-dynamics.  \smallskip

\end{document}